\documentclass[10pt]{amsart}
\usepackage{amsmath,amscd,amsthm,amsxtra}
\usepackage{epsfig,graphics,color,colortbl}
\usepackage{amssymb,latexsym}
\usepackage{mathrsfs,upgreek}
\usepackage[poly,all]{xy}
\usepackage{hyperref}
\usepackage{tikz-cd}
\usepackage{ytableau}
\usepackage[draft]{todonotes}
\usepackage{stmaryrd}

\newtheorem{thm}{\bf Theorem}[section]

\newtheorem{prop}[thm]{\bf Proposition}
\newtheorem{cor}[thm]{\bf Corollary}
\newtheorem{lem}[thm]{\bf Lemma}
\newtheorem{rem}[thm]{\bf Remark}
\newtheorem{ex}[thm]{\bf Example}
\newtheorem{conj}[thm]{\bf Conjecture}

\numberwithin{equation}{section}

\newcommand{\mc}{\mathcal}
\newcommand{\mf}{\mathfrak}

\newcommand{\ov}{\overline}

\newcommand{\W}{\mc{W}}
\newcommand{\cP}{\mathscr{P}}

\newcommand{\be}{{\bf e}}
\newcommand{\bm}{{\bf m}}

\newcommand{\Z}{\mathbb{Z}}
\newcommand{\Q}{\mathbb{Q}}
\newcommand{\C}{\mathbb{C}}

\newcommand{\e}{\epsilon}
\newcommand{\ve}{\varepsilon}

\newcommand{\te}{\widetilde{e}}
\newcommand{\tf}{\widetilde{f}}

\newcommand{\La}{\Lambda}

\newcommand{\la}{\lambda}

\newcommand{\hf}{\frac{1}{2}}

\newcommand{\tl}[1]{\substack{\scalebox{0.8}{#1}}}

\newcommand{\ot}{\otimes}
\newcommand{\ep}{\epsilon}
\newcommand{\ket}[1]{|#1\rangle}

\newcommand{\mb}{\bold{m}}
\newcommand{\eb}{\bold{e}}
\newcommand{\ib}{\bold{i}}
\newcommand{\geh}{\mathfrak{g}}

\begin{document}
\title[]
{Higher level $q$-oscillator representations for $U_q(C_n^{(1)}),U_q(C^{(2)}(n+1))$
and $U_q(B^{(1)}(0,n))$}

\author{JAE-HOON KWON}

\address{Department of Mathematical Sciences and Research Institute of Mathematics, Seoul National University, Seoul 08826, Korea}
\email{jaehoonkw@snu.ac.kr}

\author{MASATO OKADO}

\address{Department of Mathematics, Osaka City University, Osaka 558-8585, Japan}
\email{okado@sci.osaka-cu.ac.jp}

\keywords{quantum affine superalgebra; $q$-oscillator representation; $R$ matrix; crystal base}

\thanks{J.-H.K. is supported by the National Research Foundation of Korea(NRF) grant funded by the Korea government(MSIT) (No.\,2019R1A2C108483311 and 2020R1A5A1016126). 
M.O. is supported by Grants-in-Aid for Scientific Research No.\,19K03426. This work was partly supported by Osaka City University Advanced Mathematical Institute (MEXT Joint Usage/Research Center on Mathematics and Theoretical Physics JPMXP0619217849).
}

\begin{abstract}
We introduce higher level $q$-oscillator representations for the quantum affine 
(super)algebras of type $C_n^{(1)},C^{(2)}(n+1)$ and $B^{(1)}(0,n)$. 
These representations are constructed by applying the fusion procedure to the level one $q$-oscillator 
representations which were obtained through the studies of the tetrahedron equation.
We prove that these higher level $q$-oscillator representations are irreducible. 
For type $C_n^{(1)}$ and $C^{(2)}(n+1)$, we compute their characters explicitly in terms of Schur polynomials.
\end{abstract}

\maketitle
\setcounter{tocdepth}{1}

\section{Introduction}

Let $\geh$ be an affine Lie algebra and $U_q(\geh)$ the Drinfeld-Jimbo quantum
group (without derivation) associated to it. For a node $r$ of the Dynkin diagram of $\geh$ except 0
and a positive integer $s$, there exists a family of finite-dimensional $U_q(\geh)$-modules $W^{r,s}$
called Kirillov-Reshetikhin modules. They have distinguished properties. One of them
is the existence of crystal bases in Kashiwara's sense (see \cite{CH,HKOTT,OS} 
and references therein).

\begin{table}[h]
\small{
\begin{alignat*}{2}
B_n^{(1)}\;
&\vcenter{\xymatrix@R=1ex{
*{\circ}<3pt> \ar@{-}[dr]^<{0} \\
& *{\circ}<3pt> \ar@{-}[r]_<{2} 
& {} \ar@{.}[r]&{}  \ar@{-}[r]_>{\,\,\,\,n-1} &
*{\circ}<3pt> \ar@{=}[r] |-{\scalebox{2}{\object@{>}}}& *{\circ}<3pt>\ar@{}_<{n} \\
*{\circ}<3pt> \ar@{-}[ur]_<{1}}} & 
B^{(1)}(0,n)\;
&\vcenter{\xymatrix@R=1ex{
*{\circ}<3pt> \ar@{=}[r] |-{\scalebox{2}{\object@{>}}}_<{0} 
&*{\circ}<3pt> \ar@{-}[r]_<{1} 
& {} \ar@{.}[r]&{}  \ar@{-}[r]_>{\,\,\,\,n-1} &
*{\circ}<3pt> \ar@{=}[r] |-{\scalebox{2}{\object@{>}}}
& *{\bullet}<3pt>\ar@{}_<{n}}} \\
D_n^{(1)}\;
&\vcenter{\xymatrix@R=1ex{
*{\circ}<3pt> \ar@{-}[dr]^<{0}&&&&&*{\circ}<3pt> \ar@{-}[dl]^<{n-1}\\
& *{\circ}<3pt> \ar@{-}[r]_<{2} 
& {} \ar@{.}[r]&{} \ar@{-}[r]_>{\,\,\,n-2} &
*{\circ}<3pt> & \\
*{\circ}<3pt> \ar@{-}[ur]_<{1}&&&&&
*{\circ}<3pt> \ar@{-}[ul]_<{n}}} &
C_n^{(1)}\;
&\vcenter{\xymatrix@R=1ex{
*{\circ}<3pt> \ar@{=}[r] |-{\scalebox{2}{\object@{>}}}_<{0} 
&*{\circ}<3pt> \ar@{-}[r]_<{1} 
& {} \ar@{.}[r]&{}  \ar@{-}[r]_>{\,\,\,\,n-1} &
*{\circ}<3pt> \ar@{=}[r] |-{\scalebox{2}{\object@{<}}}
& *{\circ}<3pt>\ar@{}_<{n}}} \\
D_{n+1}^{(2)}\;
&\vcenter{\xymatrix@R=1ex{
*{\circ}<3pt> \ar@{=}[r] |-{\scalebox{2}{\object@{<}}}_<{0} 
&*{\circ}<3pt> \ar@{-}[r]_<{1} 
& {} \ar@{.}[r]&{}  \ar@{-}[r]_>{\,\,\,\,n-1} &
*{\circ}<3pt> \ar@{=}[r] |-{\scalebox{2}{\object@{>}}}
& *{\circ}<3pt>\ar@{}_<{n}}} & 
C^{(2)}(n+1)\;
&\vcenter{\xymatrix@R=1ex{
*{\bullet}<3pt> \ar@{=}[r] |-{\scalebox{2}{\object@{<}}}_<{0} 
&*{\circ}<3pt> \ar@{-}[r]_<{1} 
& {} \ar@{.}[r]&{}  \ar@{-}[r]_>{\,\,\,\,n-1} &
*{\circ}<3pt> \ar@{=}[r] |-{\scalebox{2}{\object@{>}}}
& *{\bullet}<3pt>\ar@{}_<{n}}} 
\end{alignat*}
}
\vspace{-3mm}
\caption{Dynkin diagrams of $(\mf{g},\ov{\mf{g}})$}
\label{tab:Dynkin}
\end{table}

Consider the affine Lie algebras $\geh=B^{(1)}_n,D^{(1)}_n,D^{(2)}_{n+1}$, whose Dynkin
diagrams are given in the left side of Table \ref{tab:Dynkin}. 
The Kirillov-Reshetikhin modules $W^{n,1}$ corresponding to the node $n$ and the integer 1 have a particularly simple structure. 
Let $V$ be a two-dimensional vector space. Then $W^{n,1}$ can be realized as $V^{\ot n}$ with an easy description of the $U_q(\mf{g})$-action.
It is irreducible when $\geh=B^{(1)}_n,D^{(2)}_{n+1}$, but for $\geh=D^{(1)}_n$ it 
decomposes into two components; $V^{\ot n}=W^{n,1}\oplus W^{n-1,1}$. 

We can consider the quantum $R$ matrix on the tensor product of Kirillov-Reshetikhin modules. We introduce a spectral parameter $x$ to the representation $W^{n,1}$, and denote the associated representation by $W^{n,1}(x)$. 
Let $\Delta$ be the coproduct and $\Delta^{\text{op}}$ its opposite. Then the 
quantum $R$ matrix $R(x/y)$ is defined as an intertwiner of $\Delta$ and 
$\Delta^{\text{op}}$, namely, linear operator satisfying $R(x/y)\Delta(u)=
\Delta^{\text{op}}(u)R(x/y)$ for any $u\in U_q(\geh)$ on $W^{n,1}(x)\ot W^{n,1}(y)$. 
($R$ is found to depend only on $x/y$.)

In \cite{KS}, Kuniba and Sergeev initiated an attempt to obtain quantum $R$
matrices from the solution to the tetrahedron equation, a three-dimensional analogue
of the Yang-Baxter equation \cite{Z}. Let $\mathcal{L}$ be a solution of the tetrahedron
equation. It is a linear operator on $F\ot V\ot V$ where $F$ is an infinite-dimensional
vector space spanned by $\{\,\ket{m}\,|\, m\in\Z_{\ge0}\,\}$. By composing this $\mathcal{L}$
$n$ times and applying suitable boundary vectors in $F$ and $F^*$, they obtained 
linear operators on $(V^{\ot n})\ot(V^{\ot n})$ satisfying the Yang-Baxter equation.
The commuting symmetry algebras were found to be $U_q(\mf{g})$ for $\mf{g}$ in Table \ref{tab:Dynkin}. The reason they have variations is that there are two choices 
of boundary vectors in each $F$ and $F^*$ corresponding to the shapes of the 
Dynkin diagrams at each end.

Associated to the tetrahedron equation, there is yet another solution $\mathcal{R}$, which is
a linear operator on $F^{\ot3}$. In \cite{KO}, Kuniba and the second author performed
the same scheme to $\mathcal{R}$ and constructed linear operators on 
$(F^{\ot n})\ot(F^{\ot n})$. As the symmetry algebra this time, they found
$U_q(C^{(1)}_n),U_q(D^{(2)}_{n+1})$ and $U_q(A^{(2)}_{2n})$, and they called these 
representations $\W=F^{\ot n}$ $q$-oscillator representations (of level one). To be precise, there are two irreducible components $\W_+$ and $\W_-$ for type $C^{(1)}_n$, and one can think of
$\W$ as either $\W_+$ or $\W_-$.
By construction, $\W$ is a bosonic analogue of $W^{n,1}$.

The purpose of this paper is to introduce and study a higher level $q$-oscillator representation corresponding to $W^{n,s}$ for $s\geq 1$.
To construct a higher level $q$-oscillator representation, we apply the fusion construction \cite{(KMN)^2} to $\W$ as in the case of $W^{n,s}$.
There is, however, a difficulty in understanding $\W$ and its tensor power since $\W$ does not have a suitable classical limit ($q\to1$) for $D^{(2)}_{n+1}$ and $A^{(2)}_{2n}$.
So we first resolve this difficulty by considering $\W$ for these two types as $q$-oscillator representations over the quantum affine superalgebras $U_q(\ov{\mf{g}})$ with $\ov{\mf{g}}=C^{(2)}(n+1)$ and $B^{(1)}(0,n)$  given in the right side of Table \ref{tab:Dynkin}, respectively.
Note that the filled nodes in the Dynkin diagrams signify anisotropic odd simple roots. If they were not filled, then the Dynkin diagrams for $C^{(2)}(n+1)$ and $B^{(1)}(0,n)$ would be the ones for $D^{(2)}_{n+1}$ and $A^{(2)\dagger}_{2n}$, respectively, where the diagram for $A^{(2)\dagger}_{2n}$ is the same as $A^{(2)}_{2n}$ with the opposite labeling of nodes.
This step is important for our construction of higher level representations. We use the twistor on quantum covering groups \cite{CFLW}, which provides explicit connection between representations of $U_q(D_{n+1}^{(2)})$ (resp.~$U_q(A_{2n}^{(2)\dagger})$) and $U_q(C^{(2)}(n+1))$ (resp.~$B^{(1)}(0,n)$).

We then investigate the quantum $R$ matrices for $\W(x)\ot\W(y)$ and apply the fusion construction \cite{(KMN)^2} to obtain a higher level representation $\W^{(s)}$ for $s\in\Z_{>0}$ and
$\ov{\mf{g}}=C_n^{(1)}, C^{(2)}(n+1),B^{(1)}(0,n)$ in Table \ref{tab:Dynkin}.
We prove that $\W^{(s)}$ is an irreducible $U_q(\ov{\mf{g}})$-module. Moreover, we have a stronger result when $\ov{\mf{g}}=C_n^{(1)}, C^{(2)}(n+1)$. In this case, we show that $\W^{(s)}$ is classically irreducible, that is, irreducible as a representation of the subalgebra corresponding to the nodes $\{1,\dots,n\}$ in Table \ref{tab:Dynkin}, and show that its classical limit is isomorphic to an irreducible highest weight oscillator representation of type $C_n$ and $B(0,n)$ (see \cite{CKW,K18} and references therein). We compute the character of $\W^{(s)}$ explicitly, which is given by a multiplicity-free sum of Schur polynomials. We also give a conjectural character formula of $\W^{(s)}$ for $\ov{\mf{g}}=B^{(1)}(0,n)$.

It is rather surprising that the $U_q(\mf{g})$-module $W^{n,s}$ corresponding to the $U_q(\ov{\mf{g}})$-module $\W^{(s)}$ is also classically irreducible when $(\mf{g},\ov{\mf{g}})=(D_n^{(1)},C_n^{(1)})$ and $(D_{n+1}^{(2)},C^{(2)}(n+1))$.
We would like to point out that this correspondence between $W^{n,s}$ and $\W^{(s)}$ also occurs in the context of super duality \cite{CLW} as representations of finite-dimensional simple Lie (super)algebras after a classical limit. 
 
The theory of super duality is an equivalence between certain parabolic Bernstein-Gelfand-Gelfand categories of classical Lie (super)algebras of infinite-rank. As a special case of this duality, this yields an equivalence between the categories for $\mc{G}_\infty$ and $\ov{\mc{G}}_\infty$, where $(\mc{G}_\infty,\ov{\mc{G}}_\infty)=(B_\infty, B(0,\infty)),(D_\infty, C_\infty)$ with Dynkin diagrams given in Table \ref{tab:Dynkin of infinite rank} \cite{CL,CLW}.
Let $\mc{G}_n$ and $\ov{\mc{G}}_n$ denote the subalgebras of $\mc{G}_\infty$ and $\ov{\mc{G}}_\infty$ of finite rank $n$, respectively.
Let $V_\infty$ be a given integrable highest weight $\mc{G}_\infty$-module.
Under this equivalence, it corresponds to an irreducible highest weight $\ov{\mc{G}}_\infty$-module, say $W_\infty$, called an oscillator representation. By applying a truncation functor to $V_\infty$ and $W_\infty$, we also obtain irreducible representations $V_n$ and $W_n$ of $\mc{G}_n$ and $\ov{\mc{G}}_n$, respectively.
\begin{table}[h]
\small{
\begin{alignat*}{2}
B_\infty\ 
&
\vcenter{\xymatrix@R=1ex{
 {} \ar@{.}[r] 
&  \ar@{-}[r]_<{} 
&*{\circ}<3pt> \ar@{-}[r]_<{} 
&*{\circ}<3pt> \ar@{-}[r]_<{} 
&*{\circ}<3pt> \ar@{=}[r] |-{\scalebox{2}{\object@{>}}}_<{}  
&*{\circ}<3pt>\ar@{}   
}} 
& 
B(0,\infty)\ 
&
\vcenter{\xymatrix@R=1ex{
 {} \ar@{.}[r] 
&  \ar@{-}[r]_<{} 
&*{\circ}<3pt> \ar@{-}[r]_<{} 
&*{\circ}<3pt> \ar@{-}[r]_<{} 
&*{\circ}<3pt> \ar@{=}[r] |-{\scalebox{2}{\object@{>}}}_<{}  
&*{\bullet}<3pt>\ar@{}   
}} \\
D_\infty\ 
&\vcenter{\xymatrix@R=1ex{
&&&&& *{\circ}<3pt>\ar@{-}[dl]_<{} & \\
  \ar@{.}[r] 
&  \ar@{-}[r]   
&*{\circ}<3pt> \ar@{-}[r]  
&*{\circ}<3pt> \ar@{-}[r] & *{\circ}<3pt> &  \\
&&&&&  *{\circ}<3pt>\ar@{-}[ul]_<{} & 
 }} &
C_\infty \ 
&
\vcenter{\xymatrix@R=1ex{
 {} \ar@{.}[r] 
&  \ar@{-}[r]_<{} 
&*{\circ}<3pt> \ar@{-}[r]_<{} 
&*{\circ}<3pt> \ar@{-}[r]_<{} 
&*{\circ}<3pt> \ar@{=}[r] |-{\scalebox{2}{\object@{<}}}_<{}  
&*{\circ}<3pt>\ar@{}   
}}
\end{alignat*}
}
\caption{Dynkin diagrams of $(\mc{G}_\infty,\ov{\mc{G}}_\infty)$}
\label{tab:Dynkin of infinite rank}
\end{table}
Let $(\mf{g},\ov{\mf{g}})$ be one of the pairs in Table \ref{tab:Dynkin}, and suppose that $\mc{G}_n$ and $\ov{\mc{G}}_n$ be the subalgebra of $\mf{g}$ and $\ov{\mf{g}}$ corresponding to $\{1,\dots,n\}$, respectively.
Then we observe that when $V_n$ is the classical limit of $W^{n,s}$, then the corresponding $W_n$ is the classical limit of $\W^{(s)}$ when $(\mf{g},\ov{\mf{g}})=(D_n^{(1)},C_n^{(1)})$ and $(D_{n+1}^{(2)},C^{(2)}(n+1))$.
Conjecture \ref{conj:KR} on $\W^{(s)}$ for $\ov{\mf{g}}=B^{(1)}(0,n)$ is based on this observation in case of $(\mf{g},\ov{\mf{g}})=(B_n^{(1)},B^{(1)}(0,n))$, which is true for $s=2$.
We strongly expect that there is a quantum affine analogue of super duality which relates the category of finite-dimensional $U_q(\mf{g})$-modules and a suitable category of infinite-dimensional $U_q(\ov{\mf{g}})$-modules including the $q$-oscillator representations, and hence explains the correspondence in this paper (cf.~\cite{KL}).

The paper is organized as follows: In Section \ref{sec:quantum superalgebra}, we briefly review the notion of quantum superalgebras. In Section \ref{sec:level one osc}, we construct a level one $q$-oscillator representation $\W$ of $U_q(\ov{\mf{g}})$ and study some of its properties including the crystal base. In Section \ref{sec:R matrix}, we consider the quantum $R$ matrix on $\W(x)\ot \W(y)$ and apply the fusion construction to define $\W^{(s)}$. In Section \ref{sec:main result}, we prove the classically irreducibility of $\W^{(s)}$ and give its character formula when $\ov{\mf{g}}=C_n^{(1)}, C^{(2)}(n+1)$. A conjecture when $\ov{\mf{g}}=B^{(1)}(0,n)$ is also given. In Appendix \ref{app:twistor}, we explain how to construct a level one $q$-oscillator representation of $U_q(\ov{\mf{g}})$ when $\ov{\mf{g}}=C^{(2)}(n+1)$ and $B^{(1)}(0,n)$ from the one for $D_{n+1}^{(2)}$ and $A_{2n}^{(2)\dagger}$ in \cite{KO}, respectively, by using the quantum covering groups and twistor \cite{CFLW}. In Appendices \ref{app:Adagger} and \ref{app:R matrix for super}, we construct the quantum $R$ matrix on $\W(x)\ot \W(y)$ for $U_q(\ov{\mf{g}})$ from the one in \cite{KO}.
In Appendix \ref{app:irreducibility}, we prove the irreducibility of $\W^{(s)}$.

\vskip 5mm

\noindent {\bf Acknowledgements} 
Part of this work was done while the first author was visiting Osaka City University. He would like to thank Department of Mathematics in OCU for its support and hospitality.
The second author would like to thank Atsuo Kuniba  for the collaboration \cite{KO} on which this work is based. Finally, the authors thank anonymous referee for careful reading of the manuscript.

\section{Quantum superalgebras}\label{sec:quantum superalgebra}

\subsection{Variant of $q$-integer}

Throughout the paper, we let $q$ be an indeterminate. Following \cite{CFLW}, we
introduce variants of $q$-integer, $q$-factorial and $q$-binomial coefficient. 
Let $\e=\pm1$.  For $m\in\Z_{\ge 0}$, we set
\begin{equation*}
[m]_{q,\e}=\frac{(\e q)^{m}-q^{-m}}{\e q-q^{-1}}.
\end{equation*}
For $m\in\Z_{\ge0}$, set 
\begin{equation*}
[m]_{q,\e}!=[m]_{q,\e}[m-1]_{q,\e}\cdots [1]_{q,\e}\quad (m\geq 1),\quad [0]_{q,\e}!=1.
\end{equation*}
For integers $m,n$ such that $0\le n\le m$, we define
\begin{equation*}
\begin{bmatrix} m \\ n \end{bmatrix}_{q,\e}= \frac{[m]_{q,\e}!}{[n]_{q,\e}![m-n]_{q,\e}!}.
\end{equation*}
They all belong to $\Z[q,q^{-1}]$.
Let $A_0$ be the subring of $\Q(q)$ consisting of rational functions without a pole
at $q=0$. Then we have
\begin{equation*}
[m]_{q,\e}\in q^{1-m}(1+qA_0),\quad [m]_{q,\e}!\in q^{-m(m-1)/2}(1+A_0),\quad
\begin{bmatrix} m \\ n  \end{bmatrix}_{q,\e}\in q^{-n(m-n)}(1+qA_0).
\end{equation*}
We simply write $[m]=[m]_{q,1}$, $[m]!=[m]_{q,1}!$ and 
$\begin{bmatrix} m \\ n \end{bmatrix}=\begin{bmatrix} m \\ n \end{bmatrix}_{q,1}$.

\subsection{Quantum (super)algebra $U_q(sl_2)$ and $U_q(osp_{1|2})$}

The quantum (super)algebras $U_q(sl_2)$ ($\e=1$) and $U_q(osp_{1|2})$ ($\e=-1$) 
are defined as a $\Q(q)$-algebra generated by $e,f,k^{\pm1}$ satisfying the
following relations:
\begin{equation*}
kk^{-1}=k^{-1}k=1,\quad kek^{-1}=q^2e,\quad kfk^{-1}=q^{-2}f,\quad
ef-\e fe=\frac{k-k^{-1}}{q-q^{-1}}.
\end{equation*}
Set $e^{(m)}=e^m/[m]_{q,\e}!$ and $f^{(m)}=f^m/[m]_{q,\e}!$. We will use the following formula.

\begin{prop}\label{prop:commutation}
\[
e^{(m)}f^{(n)}=\sum_{j\ge0}\frac{\ep^{mn-j(j+1)/2}}{[j]_{q,\e}!}f^{(n-j)}
\left(\prod_{l=0}^{j-1}\frac{(\ep q)^{2j-m-n-l}k-q^{-2j+m+n+l}k^{-1}}{q-q^{-1}}\right)e^{(m-j)}.
\]
\end{prop}

\begin{proof}
The $U_q(sl_2)$ ($\e=1$) case is derived easily from \cite[(1.1.23)]{Kas91}.
The $U_q(osp_{1|2})$ ($\e=-1$) case can be shown by induction.
\end{proof}

\subsection{Quantum affine (super)algebras $U_q(C_n^{(1)}),U_q(C^{(2)}(n+1)),U_q(B^{(1)}(0,n))$}\label{subsec:quantum affine superalgebra}

Set $I=\{0,1,\ldots,n\}$. In this paper, we consider the following three Cartan data $(a_{ij})_{i,j\in I}$, or Dynkin 
diagrams (cf.~\cite{Kac78}), and $(d_i)_{i\in I}$ such that $d_ia_{ij}=d_ja_{ji}$ for $i,j\in I$.   \vskip 5mm

$\bullet$ $C_n^{(1)}$:
\[
\vcenter{\xymatrix@R=1ex{
*{\circ}<3pt> \ar@{=}[r] |-{\scalebox{2}{\object@{>}}}_<{0} 
&*{\circ}<3pt> \ar@{-}[r]_<{1} 
& {} \ar@{.}[r]&{}  \ar@{-}[r]_>{\,\,\,\,n-1} &
*{\circ}<3pt> \ar@{=}[r] |-{\scalebox{2}{\object@{<}}}
& *{\circ}<3pt>\ar@{}_<{n}}}
\]

\[
(a_{ij})_{i,j\in I}=
\begin{pmatrix}
2&-1&\\
-2&2&-1\\
&-1& \\
&&&\ddots\\
&&&&&-1\\
&&&&-1&2&-2\\
&&&&&-1&2
\end{pmatrix}
\]

\[(d_i)_{i\in I}=(2,1,\ldots,1,2)\]
\vskip 2mm

$\bullet$ $C^{(2)}(n+1)$:
\[
\vcenter{\xymatrix@R=1ex{
*{\bullet}<3pt> \ar@{=}[r] |-{\scalebox{2}{\object@{<}}}_<{0} 
&*{\circ}<3pt> \ar@{-}[r]_<{1} 
& {} \ar@{.}[r]&{}  \ar@{-}[r]_>{\,\,\,\,n-1} &
*{\circ}<3pt> \ar@{=}[r] |-{\scalebox{2}{\object@{>}}}
& *{\bullet}<3pt>\ar@{}_<{n}}}
\]

\[
(a_{ij})_{i,j\in I}=
\begin{pmatrix}
2&-2&\\
-1&2&-1\\
&-1& \\
&&&\ddots\\
&&&&&-1\\
&&&&-1&2&-1\\
&&&&&-2&2
\end{pmatrix}
\]

\[(d_i)_{i\in I}=\left(\tfrac12,1,\ldots,1,\tfrac12\right)\]
\vskip 2mm

$\bullet$ $B^{(1)}(0,n)$:
\[
\vcenter{\xymatrix@R=1ex{
*{\circ}<3pt> \ar@{=}[r] |-{\scalebox{2}{\object@{>}}}_<{0} 
&*{\circ}<3pt> \ar@{-}[r]_<{1} 
& {} \ar@{.}[r]&{}  \ar@{-}[r]_>{\,\,\,\,n-1} &
*{\circ}<3pt> \ar@{=}[r] |-{\scalebox{2}{\object@{>}}}
& *{\bullet}<3pt>\ar@{}_<{n}}}
\]

\[
(a_{ij})_{i,j\in I}=
\begin{pmatrix}
2&-1&\\
-2&2&-1\\
&-1& \\
&&&\ddots\\
&&&&&-1\\
&&&&-1&2&-1\\
&&&&&-2&2
\end{pmatrix}
\]

\[(d_i)_{i\in I}=(2,1,\ldots,1,\tfrac12).\]

\vskip 2mm

Let $d=\min\{\,d_i\,|\,i\in I\,\}$.
For $i\in I$, let 
$q_i=q^{d_i}$, and let $p(i)=0,1$ such that
$p(i)\equiv 2d_i\,(\text{mod}\,2)$. Set 
\begin{equation*}
[m]_i=[m]_{q_i,(-1)^{p(i)}},\quad [m]_i!=[m]_{q_i,(-1)^{p(i)}}!,\quad {m\brack k}_i={m\brack k}_{q_i,(-1)^{p(i)}},
\end{equation*}
for $0\leq k\leq m$ and $i\in I$.

For a Cartan datum $X=C_n^{(1)},C^{(2)}(n+1),B^{(1)}(0,n)$, the quantum affine (super)algebra $U_q(X)$ is defined to be the $\Q(q^d)$-algebra generated by $e_i,f_i,k_i^{\pm1}$ ($i\in I$)
with the following relations:
{\allowdisplaybreaks
\begin{align*}
&k_ik_j=k_jk_i,\quad k_ik_i^{-1}=k_i^{-1}k_i=1,\quad k_ie_jk_i^{-1}=q_i^{a_{ij}}e_j,\quad k_if_jk_i^{-1}=q_i^{-a_{ij}}f_j,\\
&e_if_j-(-1)^{p(i)p(j)}f_je_i=\delta_{ij}\frac{k_i-k_i^{-1}}{q_i-q_i^{-1}},\\
&\sum_{m=0}^{1-a_{ij}}(-1)^{m+p(i)m(m-1)/2+mp(i)p(j)}e_i^{(1-a_{ij}-m)}e_je_i^{(m)}=0\quad(i\ne j),\\
&\sum_{m=0}^{1-a_{ij}}(-1)^{m+p(i)m(m-1)/2+mp(i)p(j)}f_i^{(1-a_{ij}-m)}f_jf_i^{(m)}=0\quad(i\ne j),
\end{align*}}
where
\[
e_i^{(m)}=\frac{e_i^m}{[m]_i!}, \quad
f_i^{(m)}=\frac{f_i^m}{[m]_i!}.
\]

We define the automorphism $\tau$ of $U_q(X)$ for $X=C^{(1)}_n,C^{(2)}(n+1)$ by
\begin{gather}
\tau(k_i)=k_{n-i}^{-1},\ \
\tau(e_i)=f_{n-i},\ \ 
\tau(f_i)=e_{n-i},
\quad \quad\quad\quad\quad\quad\quad \text{if $X=C_n^{(1)}$},\label{eq:auto-1}\\
\tau(k_i)=k_{n-i}^{-1},\quad 
\tau(e_i)=(-1)^{\delta_{in}}f_{n-i},\quad
\tau(f_i)=(-1)^{\delta_{i0}}e_{n-i},
\quad  \text{if $X=C^{(2)}(n+1)$}\label{eq:auto-2},
\end{gather}
for $i\in I$ and define the anti-automorphism $\eta$ of $U_q(X)$ by
\begin{gather*}
\begin{cases}
\eta(k_i)=k_i\\ 
\eta(e_i)=(-1)^{\delta_{i0}+\delta_{in}}q_i^{-1}k^{-1}_if_i\\ 
\eta(f_i)=(-1)^{\delta_{i0}+\delta_{in}}q^{-1}_ik_ie_i
\end{cases}
\ \quad \text{if $X=C_n^{(1)},B^{(1)}(0,n)$},\label{eq:antiauto-1}\\
\begin{cases}
\eta(k_i)=k_i\\ 
\eta(e_i)=(-1)^{\delta_{in}}q_i^{-1}k_i^{-1}f_i\\
\eta(f_i)=(-1)^{\delta_{in}}q_i^{-1}k_ie_i
\end{cases}
\ \ \text{if $X=C^{(2)}(n+1)$}\label{eq:antiauto-2}
\end{gather*}
for $i\in I$. Both $\tau$ and $\eta$ are involutions.

When $X=C^{(2)}(n+1), B^{(1)}(0,n)$, let
\[U_q(X)^\sigma=U_q(X)\oplus U_q(X)\sigma\] 
denote the $\Q(q^d)$-algebra generated by $U_q(X)$ and $\sigma$,
where
\begin{equation}\label{eq:relation for sigma}
\sigma^2=1,\quad \sigma k_i= k_i\sigma,\quad \sigma e_i=(-1)^{p(i)}e_i\sigma,\quad \sigma f_i=(-1)^{p(i)}f_i\sigma \quad (i\in I).
\end{equation}
We extend $\tau$ and $\eta$ to $U_q(X)^\sigma$ by $\tau(\sigma)=\eta(\sigma)
=\sigma$.

The algebras $U_q(C^{(1)}_n)$, $U_q(C^{(2)}(n+1))^\sigma$ and $U_q(B^{(1)}(0,n))^\sigma$ have a Hopf algebra structure. 
In particular, the coproduct $\Delta$ is given by
\begin{equation}\label{eq:comult-2}
\begin{split}
\Delta(k_i)&=k_i\ot k_i,\ \ \Delta(\sigma)=\sigma\ot\sigma,\\
\Delta(e_i)&=e_i\ot \sigma^{p(i)\delta_{i0}}k_i^{-1}+\sigma^{p(i)\delta_{in}}\ot e_i,\\
\Delta(f_i)&=f_i\ot \sigma^{p(i)\delta_{i0}}+\sigma^{p(i)\delta_{in}}k_i\ot f_i,
\end{split}
\end{equation}
for $i\in I$.

\section{Level one $q$-oscillator representation}\label{sec:level one osc}

Let $\W$ be an infinite-dimensional vector space  over $\Q(q^d)$ defined by
\begin{equation*}\label{eq:W}
\W = \bigoplus_{{\bf m}}\Q(q^d)| {\bf m} \rangle,
\end{equation*}
where $\ket{{\bf m}}=\ket{m_1,\dots,m_n}$ is a basis vector parametrized by ${\bf m}=(m_1,\ldots,m_n)\in \Z_{\ge 0}^n$. Let $|{\bf m}|=\sum_{j=1}^nm_j$, and let ${\bf e}_j$ be the $j$-th standard vector in $\Z^n$ for $1\leq j\leq n$. Let $\ket{{\bf 0}}=\ket{0,\cdots,0}$. In this section, we introduce the so-called $q$-oscillator representation of level one for each algebra. 

\subsection{Type $C_n^{(1)}$}

\subsubsection{$U_q(C_n^{(1)})$-module $\W_{\pm}$}

Consider the quantum affine algebra $U_q(C_n^{(1)})$. Let 
$U_q(C_n)$ and $U_q(A_{n-1})$ be the subalgebras generated by $e_i,f_i,k^{\pm 1}_i$ for $i\in I\setminus\{0\}$ and
$i\in I\setminus\{0,n\}$, respectively.
\begin{prop}\label{prop:q-osc for C_n^{(1)}}
For a non-zero $x\in \Q(q)$, the space $\W$ admits a $U_q(C_n^{(1)})$-module structure given as follows:
{\allowdisplaybreaks
\begin{align*}
e_0 | \bm \rangle &=x q^{-1}\frac{[m_1+1][m_1+2]}{[2]} | \bm +2\be_1 \rangle,\\
f_0 | \bm \rangle &= -x^{-1}\frac{q}{[2]} | \bm  - 2\be_1 \rangle,\\
k_0 | \bm \rangle &=q^{2m_1+1}|\bm\rangle,\\
e_j | \bm \rangle &= [m_{j+1}+1]| \bm -\be_j +\be_{j+1}  \rangle,\\ 
f_j | \bm \rangle &= [m_j+1]| \bm +\be_j -\be_{j+1}  \rangle,\\ 
k_j | \bm \rangle &= q^{-m_j+m_{j+1}}|\bm\rangle,\\
e_n | \bm \rangle &= -\frac{q}{[2]}| \bm  - 2\be_n \rangle,\\ 
f_n | \bm \rangle &=q^{-1}\frac{[m_n+1][m_n+2]}{[2]}| \bm  + 2\be_n \rangle,\\ 
k_n | \bm \rangle &=q^{-2m_n-1}|\bm\rangle,
\end{align*}}
where $1\leq j\leq n-1$. Here we understand the vector on the right-hand side is zero when any of its components does not belong to $\Z_{\geq 0}$. 
\end{prop}

\begin{rem} \label{rem:tau symmetry}{\rm
For $\ket{\bm}=\ket{m_1,\dots,m_n}\in \W$, set $\tau(\ket{\mb})=\ket{m_n,\ldots,m_1}$, and extend 
linearly to any vector of $\W$. Then, when $x=1$ we have the following symmetry 
\[
\tau(u\ket{\mb})=\tau(u)\tau(\ket{\mb}),
\]
for $u\in U_q(C^{(1)}_n)$. Here the automorphism $\tau$ on $U_q(C^{(1)}_n)$ is given in \eqref{eq:auto-1}.
}
\end{rem}

\begin{rem}{\rm
This representation originally appeared in \cite[Proposition 3]{KO}. 
The presentation in Proposition \ref{prop:q-osc for C_n^{(1)}} is obtained from the one in \cite{KO} by applying
the basis change $$|\bm\rangle^{\mathrm{new}}=\frac{(q[2])^{|\mb|/2}}{\prod_{i=1}^n[m_i]!}
|\bm\rangle^{\mathrm{old}},$$ and the automorphism of $U_q(C_n^{(1)})$ sending $f_0\mapsto -f_0,e_n\mapsto
-e_n,k_i\mapsto -k_i$ for $i=0,n$ with the other generators fixed.}
\end{rem}

We assume that $\ve$ denotes $+$ or $-$. Set $\varsigma(\ve)=0$ and $1$,
when $\ve=+$ and $-$, respectively. For $m\in\Z_{\ge 0}$, let ${\rm sgn}(m)$ be $+$ and $-$ if $m$ is even and odd, respectively.
Define the subspace $\W_\ve$ of $\W$ by 
\begin{equation*}
\W_\ve=\bigoplus_{{\rm sgn}(|{\bf m}|) = \ve }\Q(q)\ket{{\bf m}}.
\end{equation*}

\begin{prop}
For a non-zero $x\in \Q(q)$, $\W_\ve$ is an irreducible $U_q(C^{(1)}_n)$-module.
\end{prop}

We denote this module by $\W_\ve(x)$, and call it a (level one) $q$-oscillator representation. We simply write $\W_\ve=\W_\ve(1)$ as a $U_q(C_n^{(1)})$-module.
Then as $U_q(A_{n-1})$-modules, the characters of $\W_\pm$ are given by 
\begin{equation*}\label{eq:ch of W_+}
\begin{split}
{\rm ch}\W_+&=\sum_{l\in 2\Z_{\ge 0}}s_{(l)}(x_1,\dots,x_n)=\frac{1}{\prod_{i=1}^n(1-x_i^2)},\\
{\rm ch}\W_-&=\sum_{l\in 1+ 2\Z_{\ge 0}}s_{(l)}(x_1,\dots,x_n)=\frac{\prod_{i=1}^n(1+x_i)-1}{\prod_{i=1}^n(1-x_i^2)},
\end{split}
\end{equation*}
where $s_{(l)}(x_1,\dots,x_n)$ is the Schur polynomial corresponding to the partition $(l)$.
Here the weight lattice for $A_{n-1}$ is identified with the $\Z$-lattice spanned by $\be_i$ for $1\leq i\leq n$, and hence the variable $x_i$ corresponds to the weight of $\be_i$.

\subsubsection{Classical limit}

Let $A$ be the localization of $\Z[q,q^{-1}]$ at $[2]=q+q^{-1}$. 
Let 
\begin{equation*}
\W_\varepsilon(x)_{A} =\sum_{{\rm sgn}(|{\bf m}|) = \ve} A|{\bf m}\rangle.
\end{equation*}
Then $\W_\ve(x)_{A}$ is invariant under $e_i$, $f_i$, $k_i$ and $\{k_i\}:=\frac{k_i-k_i^{-1}}{q_i-q_i^{-1}}$ for $i\in I\setminus\{0\}$.
Let
\begin{equation*}
\ov{\W_\ve(x)}=\W_\ve(x)_{A}\otimes_{A}\mathbb{\C},
\end{equation*}
where $\C$ is an ${A}$-module such that $f(q)\cdot c = f(1)c$ for $f(q)\in A$ and $c\in \C$.

Let $E_i$, $F_i$ and $H_i$ be the $\mathbb{C}$-linear endomorphisms on $\ov{\W_\ve(x)}$ induced from $e_i$, $f_i$ and $\{k_i\}$ for $i\in I\setminus\{0\}$.
We can check that 
they satisfy the defining relations for the universal enveloping algebra $U(C_n)$ of type $C_n$ (cf.~\cite[Chapter 5]{Ja}). Hence $\ov{\W_\ve(x)}$ becomes a $U(C_n)$-module.

\begin{lem}\label{lem:classical limit-1}
The space $\ov{\W_\ve(x)}$ is isomorphic to the irreducible highest weight $U(C_n)$-module with highest weight $-(\frac{1}{2}+\varsigma(\ve))\varpi_n$, where $\varpi_n$ is the $n$-th fundamental weight for $C_n$.
\end{lem}
\begin{proof} It is clear that $E_i(\ket{{\bf 0}} \ot 1)=0$ for all $i\in I\setminus\{0\}$. Since
\begin{equation*}
\begin{split}
H_n(|{\bf 0}\rangle \ot 1) 
& = \left(\frac{k_n-k_n^{-1}}{q_n-q_n^{-1}}|{\bf 0}\rangle\right) \ot 1  = \left(-\frac{1}{q+q^{-1}}|{\bf 0}\rangle\right) \ot 1=-\frac{1}{2}|{\bf 0}\rangle \ot 1,
\end{split}
\end{equation*}
and $H_i(|{\bf 0}\rangle \ot 1)=0$ for $1\leq i\leq n-1$, $\ov{\W_+(x)}$ is a highest weight $U(C_n)$-module with highest weight $-\frac{1}{2}\varpi_n$.
It follows from the actions of $E_i$ for $i\in I\setminus\{0\}$ that any submodule of $\ov{\W_+(x)}$ contains $|{\bf 0}\rangle \ot 1$. This implies that $\ov{\W_+({x})}$ is irreducible. 
The proof for $\W_-(x)$ is similar.
\end{proof}

\subsubsection{Polarization}

Define a symmetric bilinear form on $\W_\ve$ by
\begin{equation}\label{eq:bilinear form-1}
\begin{split}
(|{\bf m}\rangle, |{\bf m'}\rangle) 
&= \delta_{{\bf m}, {\bf m}'} \frac{q^{-\frac{1}{2} \sum_{i=1}^n m_i(m_i-1)}}
{\prod_{i=1}^n [m_i]!},
\end{split}
\end{equation}
for $|{\bf m}\rangle$, $|{\bf m}'\rangle$ with ${\bf m}=(m_1,\dots,m_n)$.  
Note that $(|{\bf m}\rangle, |{\bf m}\rangle)\in 1+qA_0$.

\begin{lem} \label{lem:polarization1}
The bilinear form in \eqref{eq:bilinear form-1} is a polarization on $\W_\ve$, that is, 
\begin{equation*}
(u v , v')=  ( v, \eta(u)v' ),
\end{equation*}
for $u\in U_q(C_n^{(1)})$ and $v,v'\in \W_\ve$.
\end{lem}
\begin{proof}  It suffices to show this when $u$ is one of the generators. If $u=k_i$, it is trivial. Let us show that 
\begin{equation}\label{eq:polarization}
(e_i |{\bf m}\rangle,|{\bf m}'\rangle) = (|{\bf m}\rangle,\eta(e_i)|{\bf m}'\rangle),
\end{equation}
for $i\in I$ and $\ket{{\bf m}}, \ket{{\bf m}'}\in \W_\ve$. The proof for $f_i$ is almost identical since \eqref{eq:bilinear form-1}
is symmetric.

{\em Case 1}. Suppose that $1\leq i\leq n-1$. 
We may assume ${\bf m}'={\bf m}-\be_i+\be_{i+1}$. The right-hand side is 
\begin{align*}
(|{\bf m}\rangle,\eta(e_i)|{\bf m}-\be_i+\be_{i+1}\rangle) &= 
(|{\bf m}\rangle,q_i^{-1}k_i^{-1}f_i|{\bf m}-\be_i+\be_{i+1}\rangle)=[m_i]q^{-1+m_i-m_{i+1}}(|{\bf m}\rangle,|{\bf m}\rangle),
\end{align*}
and the left-hand side is 
\begin{align*}
(e_i |{\bf m}\rangle,|{\bf m}-\be_i+\be_{i+1}\rangle) &
= [m_{i+1}+1](|{\bf m}-\be_i+\be_{i+1}\rangle,|{\bf m}-\be_i+\be_{i+1}\rangle)\\
&=\frac{q^{X}[m_{i+1}+1]}{[m_i-1]![m_{i+1}+1]!\prod_{j\neq i,i+1}[m_j]!}
\\
&=q^{m_i-m_{i+1}-1}[m_i](|{\bf m}\rangle,|{\bf m}\rangle),
\end{align*}
since 
\begin{equation*}
\begin{split}
X
&=-\frac{1}{2}\sum_{j\neq i, i+1}m_j(m_j-1)-\frac{1}{2}(m_i-1)(m_i-2)-\frac{1}{2}(m_{i+1}+1)m_{i+1}\\
&=-\frac{1}{2}\sum_{1\leq j\leq n}m_j(m_j-1) + m_i - m_{i+1} -1.
\end{split}
\end{equation*}
Hence \eqref{eq:polarization} holds.

{\em Case 2}. Suppose that $i=n$.
We may assume ${\bf m}'={\bf m}-2\be_n$. The right-hand side is 
\begin{align*}
(|{\bf m}\rangle,\eta(e_n)|{\bf m}-2\be_n\rangle) &= 
(|{\bf m}\rangle,-q_n^{-1}k_n^{-1}f_n|{\bf m}-2\be_n\rangle) =-q^{2m_n-2}\frac{[m_n-1][m_n]}{[2]}(|{\bf m}\rangle,|{\bf m}\rangle),
\end{align*}
and the left-hand side is 
\begin{align*}
(e_n |{\bf m}\rangle,|{\bf m}-2\be_n\rangle) 
&=- \frac{q}{[2]}(|{\bf m}-2\be_n\rangle,|{\bf m}-2\be_n\rangle)=-\frac{q}{[2]}\frac{q^{Y}}{[m_n-2]!\prod_{j\neq n}[m_j]!}(|{\bf m}\rangle,|{\bf m}\rangle)\\
&=-q^{2m_n-2}\frac{[m_n-1][m_n]}{[2]}(|{\bf m}\rangle,|{\bf m}\rangle),
\end{align*}
since 
\begin{equation*}
\begin{split}
Y
&=-\frac{1}{2}\sum_{j\neq n}m_j(m_j-1)-\frac{1}{2}(m_n-2)(m_n-3)=-\frac{1}{2}\sum_{1\leq j\leq n}m_j(m_j-1) + 2m_n-3.
\end{split}
\end{equation*}
Hence \eqref{eq:polarization} holds.

{\em Case 3}. Suppose that $i=0$.
We have to show $(e_0v,v')=(v,-q^{-2}k_0^{-1}f_0v')$.
By Remark \ref{rem:tau symmetry} and the property $(\tau(\ket{\bf m}),\tau(\ket{{\bf m}'}))=(\ket{\bf m},\ket{{\bf m}'})$,
it is equivalent to $(f_n\tau(v),\tau(v'))=(\tau(v),-q^{-2}k_ne_n\tau(v'))$. However, it is equivalent to the one proved in 
{\em Case 1}.
\end{proof}\vskip 2mm


\subsubsection{Crystal base}
Let $M$ be a $U_q(C_n^{(1)})$-module. 
For $1\leq j\leq n-1$, we assume that $e_j$ and $f_j$ are locally nilpotent on $M$, and define $\te_j$, $\tf_j$ to be the usual lower crystal operators \cite{Kas91}. 
For $i=0,n$, we introduce new operators $\te_i$ and $\tf_i$ as follows:\vskip 2mm

{\em Case 1}. Let $u\in M$ be a weight vector such that $e_n u=0$ and $k_n u = q_n^{-l} u$ for some $l\in \Z_{>0}$. Put
\begin{equation}\label{eq:u_k for f_n}
u_k := q_n^{\frac{k(k+2l-1)}{2}}f_n^{(k)} u\quad (k\ge 0).
\end{equation}
Then we define 
\begin{equation}\label{eq:f_n at q=0}
\tf_n u_k =u_{k+1},\quad \te_n u_{k+1}=u_k\quad (k\ge 0).
\end{equation}

{\em Case 2}. Let $u\in M$ be a weight vector such that $f_0 u=0$ and $k_0 u = q_0^{l} u$ for some $l\in \Z_{>0}$. Put
\begin{equation}\label{eq:u_k for e_0}
u_k := q_0^{\frac{k(k+2l-1)}{2}}e_0^{(k)} u\quad (k\ge 0).
\end{equation}
Then we define 
\begin{equation}\label{eq:e_0 at q=0}
\te_0 u_k =u_{k+1},\quad \tf_0 u_{k+1}=u_k\quad (k\ge 0).
\end{equation}

\begin{rem}{\rm
The definitions of $\te_i$ and $\tf_i$ ($i=0,n$) are based on the idea that 
\begin{equation}\label{eq:te, tf and bilinear form}
(\tf_n^k u, \tf^k_n u) \in 1+qA_0 \quad (\te_0^k u', \te^k_0 u') \in 1+qA_0  \quad (k\geq 0),
\end{equation}
for $u,u'\in \W_\ve$ such that $e_nu=0$ and $f_0u'=0$ (use Proposition \ref{prop:commutation}).
}
\end{rem}

Let $A_0$ be the subring of $\Q(q)$ consisting of functions which are regular at $q=0$.
We define an $A_0$-lattice $\mc{L}_\ve$ of $\W_\ve$ and a $\Q$-basis $\mc{B}_\ve$ of $\mc{L}_\ve/q\mc{L}_\ve$ by
\begin{equation*}
\begin{split}
\mc{L}_\ve= \bigoplus_{{\rm sgn}({\bf m})=\ve} A_0| \bm \rangle,  \quad  
\mc{B}_\ve= \{\,| \bm \rangle\!\!\pmod{q\mc{L}_\ve}\,|\,{\rm sgn}({\bf m})=\ve\,\}.
\end{split}
\end{equation*}
It is clear from \eqref{eq:bilinear form-1} that $(\mc{L}_\ve,\mc{L}_\ve)\subset A_0$, and $\mc{B}_\ve$ is an orthonormal basis of $\mc{L}_\ve/q\mc{L}_\ve$ with respect to $(\ ,\ )|_{q=0}$. The following shows that $(\mc{L}_\ve,\mc{B}_\ve)$ is a crystal base of $\W_\ve$ in the sense of \cite{Kas91} even though the action of $e_i, f_i$ on $\W_\ve$ for $i=0,n$ is not locally nilpotent.

\begin{prop}\label{prop:crystal base of W_+}
The pair $(\mc{L}_\ve,\mc{B}_\ve)$ is a crystal base of $\W_\ve$, that is, 
\begin{itemize}
\item[(1)] $\mc{L}_\ve$ is invariant under $\te_i$ and $\tf_i$ for $i\in I$,

\item[(2)] $\te_i\mc{B}_\ve\subset \mc{B}_\ve \cup\{0\}$ and $\tf_i\mc{B}_\ve\subset \mc{B}_\ve \cup\{0\}$ for $i\in I$, where we have
\begin{equation*}\label{eq:crystal of W_+}
\tf_i |\bm \rangle \equiv
\begin{cases}
| \bm +2\be_n \rangle & \text{if $i=n$},\\
| \bm +\be_i - \be_{i+1} \rangle & \text{if $m_{i+1}\geq 1$ and $1\leq i\leq n-1$},\\
| \bm -2\be_1 \rangle & \text{if $m_1\geq 2$ and $i=0$},\\
0 & \text{otherwise},
\end{cases}\ \pmod{q\mc{L}_\ve}.
\end{equation*}
\end{itemize}
\end{prop}
\begin{proof} It is enough to prove (2).

{\em Case 1}. Suppose that $1\leq i\leq n-1$. Let $|{\bf m}\rangle=|m_1,\dots,m_n\rangle\in \mc{L}_\ve$ be given with $m_{i+1}\geq 1$. 
Since $e_i|{\bf m}-m_i\be_i+m_i\be_{i+1}\rangle=0$, we have
\begin{equation*}
\begin{split}
\tf_i^{m_i}|{\bf m}-m_i\be_i+m_i\be_{i+1}\rangle 
&= \frac{f_i^{m_i}}{[m_i]!}\, |{\bf m}-m_i\be_i+m_i\be_{i+1}\rangle=|{\bf m}\rangle,
\end{split}
\end{equation*}
and hence
$\tf_i|{\bf m}\rangle = \tf_i^{m_i+1}|{\bf m}-m_i\be_i+m_i\be_{i+1}\rangle 
= | \bm +\be_i - \be_{i+1} \rangle$.

{\em Case 2}. Suppose that $i=n$. First, suppose that $m_n$ is even.
Since $e_n|{\bf m}-m_n\be_n \rangle=0$ and $k_n |{\bf m}-m_n\be_n \rangle=q^{-1}|{\bf m}-m_n\be_n \rangle$, we have
\begin{equation*}
\begin{split}
\tf_n^{\frac{m_n}{2}}|{\bf m}-m_n\be_n\rangle 
&= q^{\left(\frac{m_n}{2}\right)^2} \frac{f_n^{\frac{m_n}{2}}}{\left[\frac{m_n}{2}\right]_n!}\, 
|{\bf m}-m_n\be_n\rangle
= {(1+q^2)^{-\frac{m_n}{2}}} q^{\left(\frac{m_n}{2}\right)^2} \frac{[m_n]!}{\left[\frac{m_n}{2}\right]_n!}\, |{\bf m}\rangle,
\end{split}
\end{equation*}
and hence
\begin{equation*}
\begin{split}
\tf_n|{\bf m}\rangle 
&= {(1+q^2)^{\frac{m_n}{2}}}q^{-\left(\frac{m_n}{2}\right)^2} \frac{\left[\frac{m_n}{2}\right]_n!}{[m_n]!}\ \tf_n^{\frac{m_n}{2}+1}|{\bf m}-m_n\be_n\rangle  \\
&= {(1+q^2)^{\frac{m_n}{2}}}q^{-\left(\frac{m_n}{2}\right)^2}  \frac{\left[\frac{m_n}{2}\right]_n!}{[m_n]!}\ {(1+q^2)^{-\frac{m_n}{2}-1}}
q^{\left(\frac{m_n}{2}+1\right)^2} \frac{[m_n+2]!}{\left[\frac{m_n}{2}+1\right]_n!} |{\bf m}+2\be_n\rangle  \\
&= {(1+q^2)^{-1}}q^{\left(\frac{m_n}{2}+1\right)^2-\left(\frac{m_n}{2}\right)^2} 
 \frac{[m_n+2][m_n+1]}{\left[\frac{m_n}{2}+1\right]_n} |{\bf m}+2\be_n\rangle  \\
&\equiv  |{\bf m}+2\be_n\rangle \quad \pmod{q\mc{L}_\ve}, 
\end{split}
\end{equation*}
since
\begin{equation*}
q^{\left(\frac{m_n}{2}+1\right)^2-\left(\frac{m_n}{2}\right)^2} 
 \frac{[m_n+2][m_n+1]}{\left[\frac{m_n}{2}+1\right]_n}
 = q^{m_n+1} 
 \frac{[m_n+2][m_n+1]}{\left[\frac{m_n}{2}+1\right]_n}\in (1+qA_0).
\end{equation*}

Next, suppose that $m_n$ is odd.
Since $e_n|{\bf m}-(m_n-1)\be_n \rangle=0$ and $k_n |{\bf m}-(m_n-1)\be_n \rangle=q^{-3}|{\bf m}-(m_n-1)\be_n \rangle$, we have
\begin{equation*}
\begin{split}
\tf_n^{\frac{m_n-1}{2}}|{\bf m}-(m_n-1)\be_n\rangle 
&= q^{\left(\frac{m_n-1}{2}\right)\left(\frac{m_n+3}{2}\right)} \frac{f_n^{\frac{m_n-1}{2}}}{\left[\frac{m_n-1}{2}\right]_n!}\, |{\bf m}-(m_n-1)\be_n\rangle\\
&= {(1+q^2)^{-\frac{m_n-1}{2}}} q^{\left(\frac{m_n-1}{2}\right)\left(\frac{m_n+3}{2}\right)} \frac{[m_n]!}{\left[\frac{m_n-1}{2}\right]_n!}\, |{\bf m}\rangle,
\end{split}
\end{equation*}
and hence
\begin{equation*}
\begin{split}
&\tf_n|{\bf m}\rangle \\
&= {(1+q^2)^{\frac{m_n-1}{2}}} q^{-\left(\frac{m_n-1}{2}\right)\left(\frac{m_n+3}{2}\right)} \frac{\left[\frac{m_n-1}{2}\right]_n!}{[m_n]!}\ \tf_n^{\frac{m_n+1}{2}}|{\bf m}-(m_n-1)\be_n\rangle  \\
&= {(1+q^2)^{\frac{m_n-1}{2}}} q^{-\left(\frac{m_n-1}{2}\right)\left(\frac{m_n+3}{2}\right)} \frac{\left[\frac{m_n-1}{2}\right]_n!}{[m_n]!}\
{(1+q^2)^{-\frac{m_n+1}{2}}}q^{\left(\frac{m_n+1}{2}\right)\left(\frac{m_n+5}{2}\right)} \frac{[m_n+2]!}{\left[\frac{m_n+1}{2}\right]_n!} |{\bf m}+2\be_n\rangle  \\
&= {(1+q^2)^{-1}} q^{\left(\frac{m_n+1}{2}\right)\left(\frac{m_n+5}{2}\right)-\left(\frac{m_n-1}{2}\right)\left(\frac{m_n+3}{2}\right)} 
 \frac{[m_n+2][m_n+1]}{\left[\frac{m_n+1}{2}\right]_n} |{\bf m}+2\be_n\rangle  \\
&\equiv |{\bf m}+2\be_n\rangle \quad \pmod{q\mc{L}_\ve}, 
\end{split}
\end{equation*}
since
\begin{equation*}
q^{\left(\frac{m_n+1}{2}\right)\left(\frac{m_n+5}{2}\right)-\left(\frac{m_n-1}{2}\right)\left(\frac{m_n+3}{2}\right)} 
 \frac{[m_n+2][m_n+1]}{\left[\frac{m_n+1}{2}\right]_n}
 = q^{m_n+1} 
 \frac{[m_n+2][m_n+1]}{\left[\frac{m_n+1}{2}\right]_n}\in (1+qA_0).
\end{equation*}

{\em Case 3}. Suppose that $i=0$. We can prove this case by the same arguments as in {\em Case 2} by using the automorphism $\tau$ in \eqref{eq:auto-1}.
\end{proof}

\subsection{Type $C^{(2)}(n+1)$}

\subsubsection{$U_q(C^{(2)}(n+1))$-module $\W$}
Consider the quantum affine superalgebra $U_q(C^{(2)}(n+1)$.
Let $U_q(B(0,n))$ and $U_q(A_{n-1})$ be the subalgebras of $U_q(C^{(2)}(n+1))$ generated by $ e_i, f_i, k_i^{\pm 1}$ for $i\in I\setminus\{0\}$ and $i\in I\setminus\{0,n\}$, respectively. 
We also write $U_q(B(0,n))=U_q(osp_{1|2n})$, where $osp_{1|2n}$ is the orthosymplectic Lie superalgebra corresponding to the Dynkin diagram:

\[
\vcenter{\xymatrix@R=1ex{
*{\circ}<3pt> \ar@{-}[r]_<{1} 
& {} \ar@{.}[r]&{}  \ar@{-}[r]_>{\,\,\,\,n-1} &
*{\circ}<3pt> \ar@{=}[r] |-{\scalebox{2}{\object@{>}}}
& *{\bullet}<3pt>\ar@{}_<{n}}}
\]
\vskip 2mm 

\begin{prop}\label{prop:q-osc for C^{(2)}(n+1)}
For a non-zero $x\in \Q(q^{\frac12})$, the space $\W$ admits an irreducible $U_q(C^{(2)}(n+1))^\sigma$-module structure given as follows:
{\allowdisplaybreaks
\begin{align*}
e_0\ket{\mb}&=xq^{-\frac12}[m_1+1]\ket{\mb+\eb_1},\\
f_0\ket{\mb}&=x^{-1}q^{\frac12}\ket{\mb-\eb_1},\\
k_0\ket{\mb}&=q^{m_1+\frac12}\ket{\mb},\\
e_j | \bm \rangle &= [m_{j+1}+1]\ket{\bm -\be_j +\be_{j+1}},\\ 
f_j | \bm \rangle &= [m_j+1] \ket{\bm +\be_j -\be_{j+1}},\\ 
k_j | \bm \rangle &= q^{-m_j+m_{j+1}}\ket{\bm },\\
e_n\ket{\mb}&=-q^{\frac12}\ket{\mb-\eb_n},\\
f_n\ket{\mb}&=q^{-\frac12}[m_n+1]\ket{\mb+\eb_n},\\
k_n\ket{\mb}&=q^{-m_n-\frac12}\ket{\mb},\\
\sigma \ket{\mb}& = (-1)^{|\mb|}\ket{\mb},
\end{align*}}
where $1\le j\le n-1$. 
\end{prop}
\begin{proof}
See Appendix \ref{app:proof of well-definedness}.
\end{proof}

We denote this module by $\W(x)$ and call it a (level one) $q$-oscillator representation. We simply write $\W=\W(1)$ as a $U_q(C^{(2)}(n+1))$-module.
Note that as a $U_q(A_{n-1})$-module, the character of $\W$ is given by
\begin{equation*}\label{eq:ch of W}
\begin{split}
{\rm ch}\W&=\sum_{l\in \Z_{\ge 0}}s_{(l)}(x_1,\dots,x_n)=\frac{1}{\prod_{i=1}^n(1-x_i)}. 
\end{split}
\end{equation*}

\begin{rem}{\rm
When $x=1$ we also have the following symmetry 
\[
\tau(u\ket{\mb})=\tau(u)\tau(\ket{\mb}),
\]
for $u\in U_q(C^{(2)}(n+1))$ and $\ket{\bm}\in \W$ where $\tau$ is the automorphism in \eqref{eq:auto-2} (cf. Remark \ref{rem:tau symmetry}).
}
\end{rem}

\subsubsection{Classical limit}
Let
\begin{equation*}\label{eq:A-form and classical limit of C^{(2)}(n+1)}
\W(x)_{A} =\sum_{\bm}A  |\bm\rangle,\quad \ov{\W(x)}=\W(x)_{A}\otimes_{A}\mathbb{\C},
\end{equation*}
where $A=\Z[q^{\hf},q^{-\hf}]$ and $\C$ is an ${A}$-module such that $f(q^\hf)\cdot c = f(1)c$ for $f(q^\hf)\in A$ and $c\in \C$.

One can check directly that $\W(x)_{A}$ is invariant under $e_i$, $f_i$, and $\{k_i\}$ for $i\in I\setminus\{0\}$, and the induced operators $E_i$, $F_i$, and $H_i$ on $\ov{\W(x)}$, respectively, satisfy the defining relations of $U(osp_{1|2n})$. 

\begin{lem}\label{lem:classical limit-2}
The space $\ov{\W(x)}$ is isomorphic to the irreducible highest weight $U(osp_{1|2n})$-module with highest weight $-\varpi_n$, where $\varpi_n$ is the $n$-th fundamental weight for $osp_{1|2n}$.
\end{lem}
\begin{proof}  We have
\begin{equation*}
H_n(|{\bf 0}\rangle \ot 1) 
= \left(\frac{k_n-k_n^{-1}}{q^{\hf}-q^{-\hf}}|{\bf 0}\rangle\right) \ot 1   
= \left(\frac{q^{-\hf}-q^{\hf}}{q^{\hf}-q^{-\hf}}|{\bf 0}\rangle\right) \ot 1
= - |{\bf 0}\rangle \ot 1,
\end{equation*}
and $H_i(|{\bf 0}\rangle \ot 1)=0$ for $1\leq i\leq n-1$.
By the same argument as in Lemma \ref{lem:classical limit-1}, $\ov{\W(x)}$ is an irreducible highest weight $U_q(osp_{1|2n})$-module with highest weight $-\varpi_n$.
\end{proof}

\subsubsection{Polarization}

Define a symmetric bilinear form on $\W$ by \eqref{eq:bilinear form-1}.

\begin{lem} \label{lem:polarization2}
The bilinear form in \eqref{eq:bilinear form-1} is a polarization on $\W$, that is, 
\begin{equation*}
(u v , v')=  ( v, \eta(u)v' ),
\end{equation*}
for $u\in U_q(C^{(2)}(n+1))$ and $v,v'\in \W$.
\end{lem}
\begin{proof}  Let us show 
$(e_n |{\bf m}\rangle,|{\bf m}'\rangle) = (|{\bf m}\rangle,\eta(e_n)|{\bf m}'\rangle)$
for $|{\bf m}\rangle, |{\bf m}'\rangle\in \W$ only. The proof for $e_i$ ($1\leq i\leq n-1$) is identical to Lemma \ref{eq:bilinear form-1}, and the proof for $e_0$ is obtained by using $\tau$.
We may assume ${\bf m}'={\bf m}-\be_n$. The right-hand side is 
\begin{align*}
(|{\bf m}\rangle,\eta(e_n)|{\bf m}-\be_n\rangle) &= 
(|{\bf m}\rangle,-q_n^{-1}k_n^{-1}f_n|{\bf m}-\be_n\rangle)=-q_n^{2m_n-1}[m_n](|{\bf m}\rangle,|{\bf m}\rangle),
\end{align*}
and the left-hand side is 
\begin{align*}
(e_n |{\bf m}\rangle,|{\bf m}-\be_n\rangle) 
&=- q^\hf(|{\bf m}-\be_n\rangle,|{\bf m}-\be_n\rangle)\\
&=- q^\hf \frac{q^{-\hf\sum m_i(m_i-1)}}{\prod_{i=1}^n[m_i]!}\, [m_n]\,q^{m_n-1} =-q^{m_n-\hf}[m_n](|{\bf m}\rangle,|{\bf m}\rangle).
\end{align*}
Hence the equality holds.
\end{proof}

\subsubsection{Crystal base}

Let $M$ be a $U_q(C^{(2)}(n+1))$-module. 
For $1\leq j\leq n-1$, we assume that $e_j$ and $f_j$ are locally nilpotent on $M$, and define $\te_j$, $\tf_j$ to be the usual lower crystal operators.
For $i=0,n$, we consider the operators $\te_i$ and $\tf_i$ defined in the same way as in  \eqref{eq:u_k for f_n}--\eqref{eq:e_0 at q=0} for $U_q(C_n^{(1)})$, which also satisfy \eqref{eq:te, tf and bilinear form}.

Let $A_0$ be the subring of $\Q(q^\hf)$ consisting of functions which are regular at $q^\hf=0$.
We define an $A_0$-lattice $\mc{L}$ of $\W$ and a $\Q$-basis $\mc{B}$ of $\mc{L}/q^\hf\mc{L}$ by
\begin{equation}\label{eq:crystal base for W when C^{(2)}(n+1)}
\begin{split}
\mc{L}= \bigoplus_{\bf m}A_0| \mb \rangle,  \quad  
\mc{B}= \{\,\ket{\mb}\!\!\pmod{q^\hf\mc{L}}\,\}.
\end{split}
\end{equation}
It is clear from \eqref{eq:bilinear form-1} that $(\mc{L},\mc{L})\subset A_0$, and $\mc{B}$ is an orthonormal basis of $\mc{L}/q^\hf\mc{L}$ with respect to $(\ ,\ )|_{q^\hf=0}$.

\begin{prop}\label{prop:crystal base of W-2}
The pair $(\mc{L},\mc{B})$ is a crystal base of $\W$ in the sense of Proposition \ref{prop:crystal base of W_+}, where
\begin{equation*}
\tf_i |\mb \rangle \equiv
\begin{cases}
| \mb +\eb_n \rangle  & \text{if $i=n$},\\
| \mb +\eb_i - \eb_{i+1} \rangle  & \text{if $m_{i+1}\geq 1$ and $1\leq i\leq n-1$},\\
| \mb -\eb_1 \rangle  & \text{if $m_1\geq 1$ and $i=0$},\\
0  & \text{otherwise},
\end{cases}\ \pmod{q^\hf\mc{L}}.
\end{equation*}
\end{prop}
\begin{proof} It suffices to prove the above formula for $\tf_i$ when $i=0,n$ since the other cases are proved in Proposition \ref{prop:crystal base of W_+}. Let us prove the case of $\tf_n$ only.
Recall that $[m]_n=[m]_{q^\hf,-1}$ for $m\in \Z_{\ge 0}$.

Let $|{\bf m}\rangle$ be given. 
Since $e_n|\bm-m_n\be_n\rangle=0$ and $k_n|\bm-m_n\be_n\rangle=q^{-\hf}|\bm-m_n\be_n\rangle$, we have
\begin{equation*}
\begin{split}
\tf_n^{m_n}|{\bf m}-m_n\be_n\rangle 
&= q_n^{\frac{m_n(m_n+1)}{2}} \frac{f_n^{m_n}}{\left[m_n\right]_{q_n,-1}!}\, 
|{\bf m}-m_n\be_n\rangle= q_n^{\frac{m_n(m_n-1)}{2}}\frac{[m_n]!}{\left[m_n\right]_{_{q_n,-1}}!}\, |{\bf m}\rangle,
\end{split}
\end{equation*}
and hence
\begin{equation*}
\begin{split}
\tf_n|{\bf m}\rangle 
&= q_n^{-\frac{m_n(m_n-1)}{2}} \frac{\left[m_n\right]_{q_n,-1}!}{[m_n]!}\ \tf_n^{m_n+1}|{\bf m}-m_n\be_n\rangle  \\
&= q_n^{-\frac{m_n(m_n-1)}{2}} \frac{\left[m_n\right]_{q_n,-1}!}{[m_n]!}\
q_n^{\frac{m_n(m_n+1)}{2}} \frac{[m_n+1]!}{\left[m_n+1\right]_{_{q_n,-1}}!} |{\bf m}+\be_n\rangle  \\
&\equiv q_n^{m_n} \frac{[m_n+1]}{\left[m_n+1\right]_{_{q_n,-1}}} |{\bf m}+\be_n\rangle = |{\bf m}+\be_n\rangle \quad \pmod{q^\hf\mc{L}}. 
\end{split}
\end{equation*}
\end{proof}

\subsection{Type $B^{(1)}(0,n)$}

\subsubsection{$U_q(B^{(1)}(0,n))$-module $\W$}
Consider the quantum affine superalgebra $U_q(B^{(1)}(0,n))$.
Let $U_q(B(0,n))$ (or $U_q(osp_{1|2n})$) and $U_q(A_{n-1})$ be the subalgebras of $U_q(B^{(1)}(0,n))$ generated by $k_i, e_i, f_i$ for $i\in I\setminus\{0\}$ and $i\in I\setminus\{0,n\}$, respectively. 

\begin{prop}\label{prop:q-osc for B^{(1)}(0,n)}
For a non-zero $x\in \Q(q^{\frac12})$, the space $\W$ admits an irreducible $U_q(B^{(1)}(0,n))^\sigma$-module structure given as follows:
{\allowdisplaybreaks
\begin{align*}
e_0 | \bm \rangle &=x q^{-1}\frac{[m_1+1][m_1+2]}{[2]} | \bm +2\be_1 \rangle,\\ 
f_0 | \bm \rangle &= -x^{-1}\frac{q}{[2]} | \bm  - 2\be_1 \rangle,\\ 
k_0 | \bm \rangle &=q^{2m_1+1}\ket{\bm},\\
e_j | \bm \rangle &= [m_{j+1}+1]\ket{\bm -\be_j +\be_{j+1}},\\ 
f_j | \bm \rangle &= [m_j+1] \ket{\bm +\be_j -\be_{j+1}},\\ 
k_j | \bm \rangle &= q^{-m_j+m_{j+1}}\ket{\bm},\\
e_n\ket{\mb}&=-q^{\frac12}\ket{\mb-\eb_n},\\
f_n\ket{\mb}&=q^{-\frac12}[m_n+1]\ket{\mb+\eb_n},\\
k_n\ket{\mb}&=q^{-m_n-\frac12}\ket{\mb},\\
\sigma\ket{\mb}&=(-1)^{|\mb|}\ket{\mb},
\end{align*}}
where $1\le j\le n-1$.
\end{prop}
\begin{proof}
See Appendix \ref{app:proof of well-definedness}.
\end{proof}

We also denote this module by $\W(x)$ and call it a (level one) $q$-oscillator representation. Note that the classical limit of $\W(x)$ as a $U_q(osp_{1|2n})$-module is the same as in Lemma \ref{lem:classical limit-2}.
 
\subsubsection{Polarization and crystal base}


\begin{lem}
The bilinear form in \eqref{eq:bilinear form-1} is a polarization on $\W$, that is, 
\begin{equation*}
(u v , v')=  ( v, \eta(u)v' ),
\end{equation*}
for $u\in U_q(B^{(1)}(0,n))$ and $v,v'\in \W$.
\end{lem}
\begin{proof}  All the cases are already shown in Lemmas \ref{lem:polarization1}
and \ref{lem:polarization2} since the action of $e_i$ for $0\le i<n$ (resp. $i=n$) is the
same as the one for $C^{(1)}_n$ (resp. $C^{(2)}(n+1)$).
\end{proof}

We define an $A_0$-lattice $\mc{L}$ of $\W$ and a $\Q$-basis $\mc{B}$ of $\mc{L}/q^\hf\mc{L}$ as in \eqref{eq:crystal base for W when C^{(2)}(n+1)}.
We also define the operators $\te_i$ and $\tf_i$ in the same way as in $U_q(C_n^{(1)})$ and $U_q(C^{(2)}(n+1))$.

\begin{prop}
The pair $(\mc{L},\mc{B})$ is a crystal base of $\W$ in the sense of Proposition \ref{prop:crystal base of W_+}, where
\begin{equation*}
\tf_i |\mb \rangle \equiv
\begin{cases}
| \mb +\eb_n \rangle  & \text{if $i=n$},\\
| \mb +\eb_i - \eb_{i+1} \rangle  & \text{if $m_{i+1}\geq 1$ and $1\leq i\leq n-1$},\\
| \mb -2\eb_1 \rangle  & \text{if $m_1\geq 2$ and $i=0$},\\
0  & \text{otherwise},
\end{cases}\ \pmod{q^\hf\mc{L}}.
\end{equation*}
\end{prop}
\begin{proof} It follows from Propositions \ref{prop:crystal base of W_+} and \ref{prop:crystal base of W-2}.
\end{proof}

\section{Quantum $R$ matrix and fusion construction}\label{sec:R matrix}

In this section, we review the quantum $R$ matrix and its spectral decomposition
for $U_q(X)$ and introduce higher level $q$-oscillator representations by the so-called fusion construction \cite{(KMN)^2}. 

Let $x, y\in \Q(q^d)$ be generic, and let $\W(x)$ be a level one $q$-oscillator representation of $U_q(X)$ including $\W_\ve(x)$ ($\ve=\pm$)
for type $C_n^{(1)}$. The quantum $R$ matrix $R(x,y)$ on $\W(x)\ot \W(y)$ is defined 
as a linear operator satisfying
\[
R(x,y)\Delta(a)=\Delta^\text{op}(a)R(x,y)
\]
for $a\in U_q(X)$, where $\Delta^\text{op}$ denotes
the opposite coproduct, namely, the coproduct obtained by interchanging the first and
second components in $\Delta$. If $\W(x)\ot \W(y)$ is irreducible, then $R(x,y)$
is unique up to a scalar function of $x,y$ and depends only on $z=x/y$. Let $P$ be
the linear operator on $\W(x)\ot\W(y)$ such that $P(u\ot v)=v\ot u$ and
set $\check{R}(x,y)=PR(x,y)$.  Then $\check{R}(x,y)$ maps $\W(x)\ot\W(y)$ to 
$\W(y)\ot\W(x)$.

The spectral decomposition of $\check{R}(x,y)$ can be obtained from the results in \cite{KO} and Appendix \ref{app:R matrix for super}.
For this, we only need to care about the difference between the coproduct \eqref{eq:comult-2}
and that of \cite{KO} and Appendix \ref{app:R matrix for super}, say $\bar{\Delta}$. 
More precisely, we have $\Delta=(\bar{\Delta}^\varsigma)^{\text{op}}$, where $\bar{\Delta}^\varsigma=(\varsigma\ot \varsigma) \circ \ov{\Delta}\circ \varsigma$ and $\varsigma$ is the automorphism given by $\varsigma(e_i)=e_ik_i^{-1}$, $\varsigma(f_i)=k_if_i$, and $\varsigma(k_i)=k_i$ for $i\in I$.

Hence, to translate the results in 
\cite{KO} and Appendix \ref{app:R matrix for super}, we replace $\check{R}(x,y)$ with $\check{R}(y,x)$.
Let $V_l^{\ve}$ and $V_l$ denote the irreducible submodules of $\W^{\ot 2}$ over $U_q(C_n)$ and $U_q(osp_{1|2n})$ considered in \cite{KO} and Appendix \ref{app:R matrix for super} to compute the spectral decomposition. Here we replace them with $PV_l^{\ve}$ and $PV_l$, respectively. Thus, we obtain the spectral decomposition of $\check{R}(x,y)$ as follows.

For $U_q(C^{(1)}_n)$, we have
\begin{equation} \label{eq:spectral decomp}
\check{R}_\ve(x,y)=\sum_{l\in 2\Z_{\geq 0}} \prod_{j=1}^{l/2}\frac{1-q^{4j-2}z}{z-q^{4j-2}}P^{\ve}_{l},
\end{equation}
where $\check{R}_\ve(x,y):\W_\ve(x)\ot \W_\ve(y)\rightarrow \W_\ve(y)\ot \W_\ve(x)$
for $\ve=+,-$ and $P_l^\ve$ is the projection onto $V_l^\ve$. 

For $U_q(C^{(2)}(n+1))$, we have from Proposition \ref{prop:new R matrix}
and the spectral decomposition for $U_q(D^{(2)}_{n+1})$ in \cite[Proposition 7]{KO} that
\begin{equation}\label{eq:spectral decomp-2}
\check{R}(x,y)=\sum_{l\in\Z_{\ge 0}}\prod_{j=1}^l\frac{1+(-q)^jz}{z+(-q)^j}P_l,
\end{equation}
where $P_l$ is the projection onto $V_l$.

Finally for $U_q(B^{(1)}(0,n))$,  we have from Proposition \ref{prop:new R matrix} and 
the spectral decomposition for $U_q(A_{2n}^{(2)\dagger})$ in Appendix
\ref{app:Adagger} that
\begin{equation}\label{eq:spectral decomp-3}
\check{R}(x,y)=\sum_{l\in 2\Z_{\geq 0}}
\prod_{j=1}^{l/2}\frac{1-q^{4j-1}z}{z-q^{4j-1}}P_l
+\sum_{l\in 1+2\Z_{\geq 0}}\prod_{j=0}^{(l-1)/2}\frac{1-q^{4j+1}z}{z-q^{4j+1}}P_l.
\end{equation}

Next, we explain the fusion construction. 
For $s\geq 2$, let $\mf{S}_s$ denote the group of permutations on $s$ letters generated by $s_i=(i\ i+1)$ for $1\leq i\leq s-1$.
We have $U_q(X)$-linear maps 
\begin{equation*}\label{eq:R_w}
\check{R}_w(x_1,\ldots,x_s) : 
\W(x_1)\otimes \cdots\otimes \W(x_s) \longrightarrow \W(x_{w(1)})\otimes \cdots\otimes \W(x_{w(s)}),
\end{equation*}
for $w\in \mf{S}_s$ and generic $x_1,\ldots, x_s\in \Q(q^d)$ satisfying
\begin{align*}
&\check{R}_1(x_1,\ldots,x_s) = {\rm id}_{\W(x_1)\otimes \cdots\otimes \W(x_s)},\\
&\check{R}_{s_i}(x_1,\ldots,x_s) = \left(\otimes_{j<i}{\rm id}_{\W(x_j)}\right)\otimes \check{R}(x_i/x_{i+1}) \otimes \left(\otimes_{j>i+1}{\rm id}_{\W(x_j)}\right),\\
&\check{R}_{ww'}(x_1,\ldots,x_s) = \check{R}_{w'}(x_{w(1)},\ldots,x_{w(s)})\check{R}_{w}(x_1,\ldots,x_s),
\end{align*}
for $1\le i\le s-1$ and $w, w'\in \mf{S}_s$ with $\ell(ww')=\ell(w)+\ell(w')$, where $\ell(w)$ denotes the length of $w$.
Hence we have a $U_q(X)$-linear map $\check{R}_s = \check{R}_{w_0}(x_1,\ldots,x_s)$ with $x_i=q^{d(2i-s-1)}$:
\begin{equation*}
\check{R}_s : \W(q^{d(1-s)})\ot \dots\ot  \W(q^{d(s-1)}) \longrightarrow \W(q^{d(s-1)})\ot \dots\ot \W(q^{d(1-s)}),
\end{equation*}
where $w_0$ is the longest element in $\mf{S}_s$.
Now we define a $U_q(X)$-module
\begin{equation}\label{eq:higher osc}
\W^{(s)} = {\rm Im}\check{R}_s.
\end{equation}

\begin{thm}\label{thm:irreducibility of higher level osc}
For $s\ge 2$, $\W^{(s)}$ is an irreducible $U_q(X)$-module.
\end{thm}
\begin{proof} 
See Appendix \ref{app:irreducibility}.
\end{proof}

\section{Character of higher level $q$-oscillator representation}\label{sec:main result}

In this section, we discuss the character of higher level $q$-oscillator representation.
Let $\cP$ be denote the set of partitions $\la=(\la_i)_{i\ge 1}$, and
$\cP_{\ve}$ the subset of $\la$ such that ${\rm sgn}(\la_i)=\ve$ for all $i$ with $\la_i\neq 0$
Recall that ${\rm sgn}(m)=+$ and $-$ if $m\in\Z_{\ge 0}$ is even and odd, respectively. We denote by $\ell(\la)$ the length of $\la\in \cP$. Let $s_{\la}(x_1,\dots,x_n)$ denote the Schur polynomial in $x_1,\dots,x_n$ corresponding to a partition $\la\in \cP$. 

\subsection{Type $C^{(1)}_n$}
For $s\geq 2$ and $\ve=\pm$, let $\W_\ve^{(s)}$ denote the higher level $q$-oscillator module in \eqref{eq:higher osc} corresponding to $\W_\ve$.
The following is the main result in this subsection.
 
\begin{thm}\label{thm:main-1}
For $s\geq 2$, $\W_\ve^{(s)}$ is irreducible as a $U_q(C_n)$-module, and its character is given by
\begin{equation*}
{\rm ch}\W_\ve^{(s)} =\sum_{\substack{\la \in \cP_\ve \\ \ell(\la)\leq s}}s_{\la}(x_1,\dots,x_n).
\end{equation*}
\end{thm}

\begin{cor}
The character of $\W_\ve^{(s)}$ has a stable limit for $s\geq n$ as follows:
\begin{equation*}
\begin{split}
{\rm ch}\W_\ve^{(s)} 
& =\sum_{\substack{\la \in \cP_\ve \\ \ell(\la)\leq n}}s_{\la}(x_1,\dots,x_n) \\
& = \frac{1}{\prod_{1\leq i\leq j\leq n}(1-x_ix_j)} \quad (\ve=+).
\end{split}
\end{equation*}
\end{cor}

The rest of this subsection is devoted to proving Theorem \ref{thm:main-1}.
We construct a $\Q(q)$-basis of $\W_\ve^{(2)}$, which is compatible with the action of $\check{R}(z)$, and which plays an important role in the proof of Theorem \ref{thm:main-1}. 

We note from \eqref{eq:spectral decomp} that
\begin{equation*}
\W_\ve^{(2)} = V^\ve_0= U_q(C_n)(|\varsigma(\ve)\be_n \rangle\ot |\varsigma(\ve)\be_n \rangle),
\end{equation*}
and hence it is irreducible. Moreover, we have the following character formula for $\W_\ve^{(2)}$. 

\begin{prop}\label{prop:ch of level 2}
We have
\begin{equation*}
{\rm ch}\W_\ve^{(2)}={\rm ch}V^\ve_0 = \sum_{\substack{\la \in\cP_\ve \\ \ell(\la)\leq 2}} s_{\la}(x_1,\dots,x_n).
\end{equation*}
\end{prop}
\begin{proof} Write $\W_\ve = \W_\ve(q^{\pm 1})$ for short since we may consider the action of $U_q(C_n)$ only.
Let $(\W_\ve \ot\W_\ve)_{A}$ be the $A$-span of $|\bm \rangle\ot |\bm' \rangle$ in $\W_\ve\ot\W_\ve$. Then $(\W_\ve\ot\W_\ve)_{A}$ is also invariant under $e_i$, $f_i$, $k_i$ and $\{k_i\}$ for $i\in I\setminus\{0\}$. This yields its classical limit $\ov{\W_\ve\ot\W_\ve}:=(\W_\ve\ot\W_\ve)_{A}\ot_A\C$, which is a $U(C_n)$-module. 
Also, we have as a $U(C_n)$-module
\begin{equation*}
\ov{\W_\ve\ot\W_\ve} \cong \ov{\W_\ve}\ot\ov{\W_\ve}.
\end{equation*}
By Lemma \ref{lem:classical limit-1}, $\ov{\W_\ve}$ is an irreducible highest weight module. By the theory of super duality \cite{CLW}, it belongs to a semisimple category of $U(C_n)$-module which is closed under tensor product (see \cite[Section 5.4]{K12} for more details, where we put $m=0$ there). 
Hence $\ov{\W_\ve} \ot \ov{\W_\ve}$ is semisimple, and the classical limit $\ov{V_0^\ve}$, the submodule generated by $(|\varsigma(\ve)\be_n \rangle\ot |\varsigma(\ve)\be_n \rangle)\ot 1$, is an irreducible highest weight $U(C_n)$-module with highest weight $-(1+2\varsigma(\ve))\varpi_n$. 
The character of $\ov{V^\ve_0}$ and hence $V^\ve_0$ follows from \cite[Theorem 6.1]{K18}.
\end{proof}\vskip 2mm

Let us construct a $\Q(q)$-basis of $\W^{(2)}_\ve$ which is compatible with its $U_q(A_{n-1})$-crystal base. 
For this, we find all the $U_q(A_{n-1})$-highest weight vectors in $\W^{(2)}_\ve$.


For $l\in  \Z_{\geq 0}$, let
\begin{equation}\label{eq:maximal vector of type A}
{\bf v}_l =\sum_{k=0}^l(-1)^kq^{k(k-l+1)}\begin{bmatrix} l \\ k  \end{bmatrix}^{-1}
| k\be_{n-1} +(l-k)\be_n \rangle \ot | (l-k)\be_{n-1} + k\be_n \rangle.
\end{equation}

\begin{lem}\label{lem:v_l}
For $l\in \Z_{\geq 0}$, ${\bf v}_l$ is a $U_q(A_{n-1})$-highest weight vector in $\W_\ve^{(2)}$, and   
\begin{equation*}
{\bf v}_l \equiv | l\be_n \rangle \ot | l\be_{n-1} \rangle \pmod{q\mc{L}_\ve^{\ot 2}},
\end{equation*}
where ${\rm sgn}(l)=\ve$.
\end{lem}
\begin{proof} It is straightforward to check that $e_i {\bf v}_l =0$ for $1\leq i\leq n-1$.
Next we claim that ${\bf v}_l\in \W_\ve^{(2)}$.
Note that 
\begin{equation*}
{\rm ch}\W_\ve=\sum_{l\in \varsigma(\ve)+ 2\Z_{>0}}s_{(l)}(x_1,\dots,x_n), 
\end{equation*}
and hence
\begin{equation}\label{eq:Cauchy id}
{\rm ch}\W_\ve^{\ot 2}=({\rm ch}\W_\ve)^{2} = \sum_{\substack{{\rm sgn}(|\la|)=+ \\ \ell(\la)\leq 2}}a_\la s_{\la}(x_1,\dots,x_n),
\end{equation}
where for $\la=(\la_1,\la_2)\in \cP$,
\begin{equation*}
a_\la =
\begin{cases}
\frac{\la_1-\la_2}{2} -\varsigma(\ve) & \text{if $\la_1>\la_2$}, \\
1 & \text{if $\la_1=\la_2$}.
\end{cases}
\end{equation*} 
Let $S_l$ be the $U_q(A_{n-1})$-submodule of $\W_\ve^{\ot 2}$ generated by ${\bf v}_l$.
Since the character of $S_l$ is $s_{(l^2)}(x_1,\dots,x_n)$, and
the multiplicity of $s_{(l^2)}(x_1,\dots,x_n)$ in \eqref{eq:Cauchy id} is one, it follows from Proposition \ref{prop:ch of level 2} that $S_l\subset \W_\ve^{(2)}$. This shows that ${\bf v}_l\in \W_\ve^{(2)}$.
The lemma follows from 
$q^{k(k-l+1)}\begin{bmatrix} l \\ k  \end{bmatrix}^{-1}\in q^k(1+qA_0)$.
\end{proof}\vskip 2mm

One can prove more directly that ${\bf v}_l\in \W^{(2)}_\ve$ using the following lemma.
\begin{lem}\label{lem:operator E}
Set $\mc{E}=e_{n-2}^{(2)}\cdots e_1^{(2)}e_0$, where it should be understood as 
$e_0$ when $n=2$. Then for $l\in \Z_{\ge0}$ we have
\[
(\mc{E}e_1^{(2)}\mc{E}-\frac{1}{[3]!}(e_1\mc{E})^2){\bf v}_l
=q^{-2}\frac{[2]}{[3]}([l+1][l+2])^2{\bf v}_{l+2}.
\]
\end{lem}

\begin{proof}
Denote the module $\W_\ve$ by $\W_{\ve,n}$ to signify the rank $n$
and let $\W_{\ve,n}^0$ be a linear subspace of $\W_{\ve,n}$ spanned by
the vectors $|0^{n-2},m_{n-1},m_n\rangle$.
Let $\pi:\W_{\ve,n}\rightarrow\W_{\ve,2}$ be a linear map defined by
$\pi(|{\bf m}\rangle)=|m_{n-1},m_n\rangle$, where ${\bf m}=(m_1,\ldots,m_n)$.
Then we can show by direct calculation that the following diagram commutes.
\[
\xymatrix{
(\W_{\ve,n}^0)^{\ot2} \ar[r]^{\pi^{\ot2}} \ar[d]_{\mc{E}} & \W_{\ve,2}^{\ot2} \ar[d]^{e_0}\\
(\W_{\ve,n}^0)^{\ot2} \ar[r]_{\pi^{\ot2}} & \W_{\ve,2}^{\ot2}
}
\]
This fact reduces the proof of the lemma to the case of $n=2$.

When $n=2$, one calculates
{\allowdisplaybreaks
\begin{align*}
e_0e_1^{(2)}e_0{\bf v}_l=&\sum_kc_k[l-k+1][l-k+2]\\
&\times\{q^{-2l+2k-4}[k+1][k+2]|k+2,l-k+2\rangle\ot|l-k,k\rangle\\
&+(q^{-1}[l-k+1][l-k+2]+q^{-2l-7}[k-1][k])|k,l-k+2\rangle\ot|l-k+2,k\rangle\\
&+q^{-2k}[l-k+3][l-k+4]|k-2,l-k+2\rangle\ot|l-k+4,k\rangle\}.
\end{align*}}
Here $c_k=(-1)^kq^{k(k-l+1)}\begin{bmatrix} l \\ k  \end{bmatrix}^{-1}$
and we have used the relation $q^{l-2k-2}[l-k]c_{k+1}+[k+1]c_k=0$. On the other hand,
we also get
\begin{align*}
(e_1e_0)^2{\bf v}_l=&[2]\sum_kc_k[l-k+1][l-k+2]\\
&\times\{[3]q^{-2l+2k-4}[k+1][k+2]|k+2,l-k+2\rangle\ot|l-k,k\rangle\\
&+A_k|k,l-k+2\rangle\ot|l-k+2,k\rangle\\
&+[3]q^{-2k}[l-k+3][l-k+4]|k-2,l-k+2\rangle\ot|l-k+4,k\rangle\},
\end{align*}
where
\begin{align*}
A_k=\frac{q^{l-2k}}{q-q^{-1}}&\{(1+q^{-2l-6})(q^2[k+1][l-k+2]-[k][l-k+3])\\
&-q^{-2l+2k}(1+q^{-4})([k+1][l-k+2]-q^{-4}[k][l-k+3])\}.
\end{align*}
Combining these results, we obtain
\begin{align*}
(e_0e_1^{(2)}e_0-&\frac1{[3]!}(e_1e_0)^2){\bf v}_l\\
&=\frac{[2]}{[3]}[l+1][l+2]\sum_kc_kq^{-2k-2}
[l-k+1][l-k+2]|k,l-k+2\rangle\ot|l-k+2,k\rangle\\
&=q^{-2}\frac{[2]}{[3]}([l+1][l+2])^2{\bf v}_{l+2}.
\end{align*}
\end{proof}

For $l\in \Z_{\geq 0}$ and $l'=2m\in 2\Z_{\geq 0}$, set
\begin{equation*}
{\bf v}_{l,l'} = q_n^{\frac{m(m+2l+1)}{2}}f_n^{(m)} {\bf v}_l.
\end{equation*}
Note that ${\bf v}_{l,l'}$ may not be equal to $\tf_n^m {\bf v}_l$ in the sense of \eqref{eq:f_n at q=0} since $e_n {\bf v}_{l,l'}\neq 0$ in general.

\begin{lem}\label{lem:v_{l,m} at q=0}
For $l\in \Z_{\geq 0}$ and $l'\in 2\Z_{\geq 0}$ with ${\rm sgn}(l)=\ve$, 
${\bf v}_{l,l'}$ is a $U_q(A_{n-1})$-highest weight vector in $\W_\ve^{(2)}$, and 
\begin{equation*}
{\bf v}_{l,l'} \equiv | l\be_n \rangle \ot | l\be_{n-1} +l'\be_n\rangle \pmod{q\mc{L}_\ve^{\ot 2}}.
\end{equation*}
\end{lem}
\begin{proof}  Let us assume that $l$ is even, and hence $\ve=+$, since the proof for odd $l$ is almost identical.
Since $e_j$ ($1\leq j\leq n-1$) commutes with $f_n$, it is clear that ${\bf v}_{l,l'}$ is a $U_q(A_{n-1})$-highest weight vector in $\W_\ve^{(2)}$.

Let $l'=2m$. For $0\leq c\leq l$, we have
\begin{equation*}
|c\be_{n-1} + (l-c)\be_n \rangle  \equiv 
\begin{cases}
\tf_n^{\lfloor \frac{l-c}{2}\rfloor}|c\be_{n-1}\rangle &\text{if $c$ is even},\\
\tf_n^{\lfloor \frac{l-c}{2}\rfloor}|c\be_{n-1}+\be_n\rangle &\text{if $c$ is odd},\\
\end{cases}\quad \pmod{q\mc{L}_+}.
\end{equation*}
Put $a=\lfloor \frac{l-c}{2}\rfloor$ and $b=\lfloor \frac{c}{2}\rfloor$.

{\em Case 1}. Suppose that $c$ is even. Let
\begin{equation*}
u_1 = |c\be_{n-1}\rangle,\quad u_2=|(l-c)\be_{n-1}\rangle.
\end{equation*}
We have
{\allowdisplaybreaks
\begin{align*}
&\Delta(f_n^{(m)})(\tf_n^a u_1 \ot \tf_n^b u_2) \\
&=\sum_{k=0}^mq_n^{-k(m-k)}f_n^{(m-k)}k_n^k(\tf_n^a u_1)\ot f_n^{(k)}(\tf_n^b u_2)\\
&=\sum_{k=0}^mq_n^{-k(m-k)-(\frac{1}{2}+2a)k}f_n^{(m-k)}\tf_n^{a}u_1\ot f_n^{(k)}\tf_n^{b}u_2\\
&=\sum_{k=0}^mq_n^{-k(m-k) -(\frac{1}{2}+2a)k+\frac{a^2}{2}+\frac{b^2}{2}}f_n^{(m-k)}f_n^{(a)}u_1 \ot f_n^{(k)}f_n^{(b)}u_2 \\
&=\sum_{k=0}^mq_n^{-k(m-k) -(\frac{1}{2}+2a)k+\frac{a^2}{2}+\frac{b^2}{2}}
\begin{bmatrix} m-k+a \\ a \end{bmatrix}_n\begin{bmatrix} k+b \\ b \end{bmatrix}_n
f_n^{(m-k+a)}u_1 \ot f_n^{(k+b)}u_2 \\
&=\sum_{k=0}^mq_n^{-k(m-k) -(\frac{1}{2}+2a)k+\frac{a^2}{2}+\frac{b^2}{2}-\frac{(m-k+a)^2}{2}-\frac{(k+b)^2}{2}} \\ & \hskip 5cm  \times \begin{bmatrix} m-k+a \\ a \end{bmatrix}_n\begin{bmatrix} k+b \\ b \end{bmatrix}_n
\tf_n^{m-k+a}u_1 \ot \tf_n^{k+b}u_2 \\
&=\sum_{k=0}^mf_{a,b}(q)\tf_n^{m-k+a}u_1 \ot \tf_n^{k+b}u_2.
\end{align*}}
Multiplying $q_n^{\frac{m(m+2l+1)}{2}}$ on both sides, we have $q_n^{\frac{m(m+2l+1)}{2}}f_{a,b}(q)\in q^d(1+qA_0)$, where 
\begin{equation}\label{eq:degree when c is even}
\begin{split}
&d=m(m+2l+1)-2k(m-k) -2\left(\frac{1}{2}+2a\right)k \\ 
& \hskip 1.5cm + a^2+ b^2 -(m-k+a)^2 - (k+b)^2 - 2(m-k)a -2kb\\
&=2lm+(m-k)-4ma-4kb=2lm+(m-k)-4m\left(\frac{l-c}{2}\right)-4k\left(\frac{c}{2}\right)\\
&=(m-k)+ 2c(m-k) = (2c+1)(m-k)
\end{split}
\end{equation}
since $a=\frac{l-c}{2}$ and $b=\frac{c}{2}$.\vskip 2mm

{\em Case 2}. Suppose that $c$ is odd. Let
\begin{equation*}
u_1 = |c\be_{n-1}+\be_n\rangle,\quad u_2=|(l-c)\be_{n-1}+\be_n\rangle.
\end{equation*}
We have
{\allowdisplaybreaks
\begin{align*}
&\Delta(f_n^{(m)})(\tf_n^a u_1 \ot \tf_n^b u_2) \\
&=\sum_{k=0}^mq_n^{-k(m-k)}f_n^{(m-k)}k_n^k(\tf_n^a u_1)\ot f_n^{(k)}(\tf_n^b u_2)\\
&=\sum_{k=0}^mq_n^{-k(m-k)-(\frac{3}{2}+2a)}f_n^{(m-k)}\tf_n^{(a)}u_1 \ot f_n^{(k)}\tf_n^{(b)}u_2 \\
&=\sum_{k=0}^mq_n^{-k(m-k) -(\frac{3}{2}+2a)k+\frac{a(a+2)}{2}+\frac{b(b+2)}{2}}f_n^{(m-k)}f_n^{(a)}u_1 \ot f_n^{(k)}f_n^{(b)}u_2\\
&=\sum_{k=0}^mq_n^{-k(m-k) -(\frac{3}{2}+2a)k+\frac{a(a+2)}{2}+\frac{b(b+2)}{2}}
\begin{bmatrix} m-k+a \\ a \end{bmatrix}_n\begin{bmatrix} k+b \\ b \end{bmatrix}_n
f_n^{(m-k+a)}u_1 \ot f_n^{(k+b)}u_2 \\
&=\sum_{k=0}^mq_n^{-k(m-k) -(\frac{3}{2}+2a)k+\frac{a(a+2)}{2}+\frac{b(b+2)}{2}-\frac{(m-k+a)(m-k+a+2)}{2}-\frac{(k+b)(k+b+2)}{2}} \\ & \hskip 5cm \times \begin{bmatrix} m-k+a \\ a \end{bmatrix}_n\begin{bmatrix} k+b \\ b \end{bmatrix}_n
\tf_n^{m-k+a}u_1 \ot \tf_n^{k+b}u_2 \\
&=\sum_{k=0}^mg_{a,b}(q)\tf_n^{m-k+a}u_1 \ot \tf_n^{k+b}u_2.
\end{align*}}
Multiplying $q_n^{\frac{m(m+2l+1)}{2}}$ on both sides, we have $q_n^{\frac{m(m+2l+1)}{2}}g_{a,b}(q)\in q^{d'}(1+qA_0)$, where  
\begin{equation}\label{eq:degree when c is odd}
\begin{split}
d'&= d -2k+2a+2b-2(m-k+a) - 2(k+b)\\ 
&=d - 2k-2m  \\
&= 2lm+(m-k)-4ma-4kb -2k-2m  \\
&=2lm+(m-k)-4m\left(\frac{l-c-1}{2}\right)-4k\left(\frac{c-1}{2}\right)-2k-2m \\
&=(m-k)+ 2c(m-k) = (2c+1)(m-k)
\end{split}
\end{equation}
by putting $a=\frac{l-c-1}{2}$ and $b=\frac{c-1}{2}$.
By \eqref{eq:degree when c is even}, \eqref{eq:degree when c is odd}, and Lemma \ref{lem:v_l}, we have 
$$q_n^{\frac{m(m+2l+1)}{2}}f_n^{(m)} {\bf v}_l\equiv | l\be_n \rangle \ot | l\be_{n-1} +2 m\be_n\rangle \pmod{q\mc{L}_+^{\ot 2}}.$$
\end{proof}

\begin{cor}
The set $\{\,{\bf v}_{l,l'}\,|\,l\in\Z_{\ge 0},\, l'\in 2\Z_{\ge 0},\, {\rm sgn}(l)=\ve\,\}$ is the set of $U_q(A_{n-1})$-highest weight vectors in $\W_\ve^{(2)}$.
\end{cor}
\begin{proof} 
The character of the $U_q(A_{n-1})$-submodule of $\W_\ve$ generated by ${\bf v}_{l,l'}$ is $s_\la(x_1,\dots,x_n)$ where $\la=(l'+l,l)$. 
Hence it follows from Proposition \ref{prop:ch of level 2} that there is no other $U_q(A_{n-1})$-highest weight vectors in $\W_\ve^{(2)}$.
\end{proof}\vskip 2mm

Now we define the pair $(\mc{L}_\ve^{(2)},\mc{B}_\ve^{(2)})$ as follows:
\begin{equation*}
\begin{split}
&\mc{L}_\ve^{(2)}=
\sum_{\substack{l_1\in \Z_{\geq 0} \\ {\rm sgn}(l_1)=\ve}}
\sum_{\substack{l_2\in 2\Z_{\geq 0}}}
\sum_{\substack{r\geq 0\\ 1\leq i_1,\ldots,i_r\leq n-1}}
A_0 \tf_{i_1}\ldots \tf_{i_r} {\bf v}_{l_1,l_2},\\
&\mc{B}_\ve^{(2)}=
\Big\{\,\tf_{i_1}\ldots \tf_{i_r} {\bf v}_{l_1,l_2} \!\!\!\pmod{q\mc{L}_\ve^{(2)}}\,\Big|\, \\
&\quad\quad\quad\quad   l_1\in \Z_{\geq 0},\ {\rm sgn}(l_1)=\ve,\ l_2\in 2\Z_{\geq 0},\ 
 r\geq 0,\ 1\leq i_1,\ldots,i_r\leq n-1\,\Big\}\setminus\{0\}.
\end{split}
\end{equation*}

\begin{prop}\label{prop:crystal base of W^{(2)}_+}
We have 
\begin{itemize}
\item[(1)] $\mc{L}_\ve^{(2)}\subset \mc{L}_\ve^{\ot 2}$ and $\mc{B}_\ve^{(2)}\subset \mc{B}_\ve^{\ot 2}$,

\item[(2)] $(\mc{L}_\ve^{(2)},\mc{B}_\ve^{(2)})$ is a $U_q(A_{n-1})$-crystal base of $\W_\ve^{(2)}$.
\end{itemize}
\end{prop}
\begin{proof} (1) By Proposition \ref{prop:crystal base of W_+}, $\mc{L}_\ve^{\ot 2}$ is a crystal base of $\W_\ve^{\ot 2}$ as a $U_q(A_{n-1})$-module, hence it is invariant under $\tf_i$ for $1\leq i\leq n-1$. Since ${\bf v}_{l_1,l_2}\in \mc{L}_\ve^{\ot 2}$ by Lemma \ref{lem:v_{l,m} at q=0}, we have $\tf_{i_1}\ldots \tf_{i_r} {\bf v}_{l_1,l_2}\in \mc{L}_\ve^{\ot 2}$ and hence $\tf_{i_1}\ldots \tf_{i_r} {\bf v}_{l_1,l_2}\in \mc{B}_\ve^{\ot 2}$ $\pmod{q\mc{L}_\ve^{\ot 2}}$.

(2) By definition of $(\mc{L}_\ve^{(2)},\mc{B}_\ve^{(2)})$ and Lemma \ref{lem:v_{l,m} at q=0}, $(\mc{L}_\ve^{(2)},\mc{B}_\ve^{(2)})$ is a $U_q(A_{n-1})$-crystal base of the submodule $V$ of $\W_\e^{(2)}$ generated by ${\bf v}_{l_1,l_2}$ for $l_1,l_2$. 
On the other hand, we have $V=\W_\ve^{(2)}$ by Proposition \ref{prop:ch of level 2}. Hence $(\mc{L}_\ve^{(2)},\mc{B}_\ve^{(2)})$ is a $U_q(A_{n-1})$-crystal base of $\W_\ve^{(2)}$.
\end{proof}\vskip 2mm

For $|{\bf m}\rangle=| m_{1},\ldots,m_{n}\rangle\in \W$, let $T({\bf m})$ denote the semistandard tableau of shape $(|{\bf m}|)$, a single row of length $|{\bf m}|$, with letters in $\{\,\ov{n}<\dots<\ov{1}\,\}$ such that the number of occurrences of $\ov{i}$ is $m_i$ for $1\leq i\leq n$.

Suppose that $|{\bf m}_1\rangle, \dots, |{\bf m}_s\rangle$ are given such that $|{\bf m}_1|\leq \dots \leq |{\bf m}_s|$. Let $\la=(|{\bf m}_s|\geq \dots \geq |{\bf m}_1|)$, which is a partition or its Young diagram, and $\la^\pi$ denote the Young diagram obtained by $180^\circ$-rotation of $\la$.
We denote by $T({\bf m}_1,\dots,{\bf m}_s)$ the row-semistandard tableau of shape $\la^\pi$, whose $j$-th row from the top is equal to $T({\bf m}_j)$ for $1\leq j\leq s$.

\begin{ex}\label{ex:T(m_1,m_2)}
{\rm
Suppose that $n=5$. If $|{\bf m}_1\rangle = | 2,1,0,0,2 \rangle$ and $|{\bf m}_2\rangle = | 0,1,2,3,1 \rangle$, then

\begin{equation*}
T({\bf m}_1,{\bf m}_2) = \
\ytableausetup{mathmode, boxsize=1em} 
\begin{ytableau}
\none & \none &  \ov{\tl 5} & \ov{\tl 5} & \ov{\tl 2} & \ov{\tl 1} & \ov{\tl 1} \\
\ov{\tl 5} & \ov{\tl 4} & \ov{\tl 4} & \ov{\tl 4} & \ov{\tl 3} & \ov{\tl 3} & \ov{\tl 2} \\
\end{ytableau}\ .
\end{equation*}\vskip 2mm
}
\end{ex}


\begin{prop}\label{prop:crystal of level 2}
We have
\begin{equation*}
\mc{B}_\ve^{(2)} =\left\{\,|{\bf m}_1\rangle \ot |{\bf m}_2\rangle\!\!\!\pmod{q\mc{L}_\ve^{(2)}} \,\Big|\,|{\bf m}_1|\leq |{\bf m}_2|,\ \text{\rm $T({\bf m}_1,{\bf m}_2)$ is semistandard}\,\right\}.
\end{equation*}
\end{prop}
\begin{proof} For $l_1\in\Z_{\geq 0}$ and $l_2\in 2\Z_{\geq 0}$ with ${\rm sgn}(l_1)=\ve$, let us identify ${\bf v}_{l_1,l_2}=| l_1\be_n \rangle \ot | l_1\be_{n-1} +l_2\be_n\rangle$ in $\mc{B}_\ve^{\ot 2}$ with the pair $(l_1\be_n,l_1\be_{n-1} +l_2\be_n)$ and the connected component of ${\bf v}_{l_1,l_2}$ as a $U_q(A_{n-1})$-crystal with the set of corresponding set of pairs $({\bf m}_1,{\bf m}_2)'s$. 
Then $T({\bf v}_{l_1,l_2})$ is the semistandard tableau of shape $(l_1+l_2,l_1)^\pi$. Since $\te_j{\bf v}_{l_1,l_2}=0$ for $1\leq j\leq n-1$, $T({\bf v}_{l_1,l_2})$ is the tableau of highest weight and the set
\begin{equation}\label{eq:description of B_+^2}
\left\{\,T\left(\tf_{i_1}\ldots \tf_{i_r} {\bf v}_{l_1,l_2}\right)\,\Big|\, r\geq 0,\ 1\leq i_1,\ldots,i_r\leq n-1\,\right\}\setminus\{0\}
\end{equation}
is equal to the set of semistandard tableau of shape $(l_1+l_2,l_1)^\pi$ with letters in $\{\,\ov{n}<\dots<\ov{1}\,\}$.
\end{proof}

Let $|{\bf m}_1\rangle, |{\bf m}_2\rangle\in \mc{B}_\ve$ be given with $|{\bf m}_1|=d_1$ and  $|{\bf m}_2|=d_2$, let $P({\bf m}_1,{\bf m}_2)$ denote a unique semistandard tableau of shape $\mu^\pi$ for some partition $\mu$, which is equivalent to $|{\bf m}_1\rangle \ot |{\bf m}_2\rangle$ as an element of $U_q(A_{n-1})$-crystals. 
This is equivalent to saying that if we read the row word of $T({\bf m}_1)$ from left to right, and then apply the Schensted's column insertion to $T({\bf m}_2)$ in a reverse way starting from the right-most column, then the resulting tableau is $P({\bf m}_1,{\bf m}_2)$.
So $P({\bf m}_1,{\bf m}_2)$ is of shape $(d'_2,d'_1)^\pi$ for some $d'_1\leq d'_2$ with $d'_1\leq d_1$, $d'_2\geq d_2$, and $d'_1+d'_2=d_1+d_2$. In particular, $P({\bf m}_1,{\bf m}_2)=T({\bf m}_1,{\bf m}_2)$ if $d_1\leq d_2$ and $|{\bf m}_1\rangle \ot |{\bf m}_2\rangle\in \mc{B}_\ve^{(2)}$.


\begin{ex}{\rm
Let $|{\bf m}_1\rangle, |{\bf m}_2\rangle$ be as in Example \ref{ex:T(m_1,m_2)}. Then

\begin{equation*}
P({\bf m}_1,{\bf m}_2)=\
\ytableausetup{mathmode, boxsize=1em} 
\begin{ytableau}
\none & \none & \none & \none &\none & \none &\none & \none &  \ov{\tl 5} & \ov{\tl 5} \\
\ov{\tl 5} & \ov{\tl 4} & \ov{\tl 4} & \ov{\tl 4} & \ov{\tl 3} & \ov{\tl 3} & \ov{\tl 2} & \ov{\tl 2} & \ov{\tl 1} & \ov{\tl 1}\\
\end{ytableau}\ .
\end{equation*}\vskip 2mm}
\end{ex}

Let $l_1\in\Z_{\geq 0}$ and $\l_2\in 2\Z_{\geq 0}$ be given with ${\rm sgn}(l_1)=\ve$. Put $\la=(\la_1,\la_2)=(l_1+l_2,l_1)$.
Let $SST(\la^\pi)$ be the set of semistandard tableaux of shape $\la^\pi$ with letters in $\{\,\ov{n}<\dots<\ov{1}\,\}$.
For each $T\in SST(\la^\pi)$, we choose $i_1,\dots,i_r\in I\setminus\{0,n\}$ such that
$T=T\left(\tf_{i_1}\ldots \tf_{i_r} {\bf v}_{l_1,l_2}\right)$ (see \eqref{eq:description of B_+^2}), and define
\begin{equation}\label{eq:def of v_T}
{\bf v}_T = \tf_{i_1}\ldots \tf_{i_r} {\bf v}_{l_1,l_2} \in \mc{L}_\ve^{(2)}.
\end{equation}
By Proposition \ref{prop:crystal of level 2}, we have a $\Q(q)$-basis of $\W_\ve^{(2)}$ 
\begin{equation}\label{eq:basis of level 2 KR}
\bigsqcup_{{\substack{\la \in\cP_\ve \\ \ell(\la)\leq 2}}}
\left\{\,{\bf v}_T\,|\,T\in SST(\la^\pi)\,\right\}.
\end{equation}

\begin{lem}\label{eq:expansion of basis vector}
For $T\in SST(\la^\pi)$, we have
\begin{equation*}
{\bf v}_T= |{\bf m}_1\rangle \ot |{\bf m}_2\rangle +\sum_{{\bf m}'_1, {\bf m}'_2}c_{{\bf m}'_1, {\bf m}'_2} |{\bf m}'_1\rangle\ot |{\bf m}'_2\rangle,
\end{equation*}
where $P({\bf m}_1,{\bf m}_2)=T$, $P({\bf m}'_1,{\bf m}'_2)$ is of shape $\mu^\pi$ with $\mu\vartriangleright \la$ and $\mu\neq \la$, and $c_{{\bf m}'_1, {\bf m}'_2}\in qA_0$.
%
%
Here $\vartriangleright$ denotes a dominance order on partitions, that is, $\mu_1> \la_1$, and $\mu_1+\mu_2=\la_1+\la_2$.
\end{lem}
\begin{proof} By Lemmas \ref{lem:v_l} and \ref{lem:v_{l,m} at q=0} (see also their proofs), we observe that  
\begin{equation}\label{eq:expansion of maximal vector}
{\bf v}_{l_1,l_2} = | l_1\be_n \rangle\ot |l_1\be_{n-1}+l_2\be_n\rangle +
\sum_{} c_{x,y,z,w}|x\be_{n-1}+y\be_n\rangle\ot |z\be_{n-1}+w\be_n\rangle,
\end{equation}
where the sum is over $(x,y,z,w)$ such that
\begin{itemize}
\item[(1)] $0< x  \leq  l_1$ with $x+z=l_1$,

\item[(2)] $y\geq z$, $w\geq x$ with $y+w=l_1+l_2$,

\item[(3)] $c_{x,y,z,w}\in qA_0$.
\end{itemize}
We may regard $| l_1\be_n \rangle\ot |l_1\be_{n-1}+l_2\be_n\rangle$ as the case when $(x,y,z,w)=(0,l_1,l_1,l_2)$. Then it is not difficult to see that if the shape of $P(x\be_{n-1}+y\be_n,z\be_{n-1}+w\be_n)$ is $\mu^\pi=(\mu_1,\mu_2)^\pi$, then $\mu_2=z=l_1-x\leq l_1$ and hence
$\mu \vartriangleright \la$,
and $\mu\neq \la$ when $x>0$. 

Let $i_1,\dots,i_r\in I\setminus\{0,n\}$ be the sequence in \eqref{eq:def of v_T}.
By the tensor product rule of crystals, we have 
\begin{equation}\label{eq:expansion of maximal vector-2}
\tf_{i_1}\ldots \tf_{i_r} (|x\be_{n-1}+y\be_n\rangle\ot |z\be_{n-1}+w\be_n\rangle)
=\sum_{{\bf m}_1,{\bf m}_2}c_{{\bf m}_1,{\bf m}_2}|{\bf m}_1\rangle \ot |{\bf m}_2\rangle,
\end{equation}
where the sum is over ${\bf m}_1,{\bf m}_2$ such that
\begin{itemize}
\item[(1)] $c_{{\bf m}_1,{\bf m}_2}(q)\in A_0$ such that
\begin{equation*}
c_{{\bf m}_1,{\bf m}_2}(0)=
\begin{cases}
1 &\text{if $|{\bf m}_1\rangle \ot |{\bf m}_2\rangle=\tf_{i_1}\ldots \tf_{i_r}(|x\be_{n-1}+y\be_n\rangle\ot |z\be_{n-1}+w\be_n\rangle)$,}\\
0 &\text{otherwise}.
\end{cases}
\end{equation*}

\item[(2)] $\nu \vartriangleright \la$ and $\nu\neq \la$, where $\nu^\pi$ is the shape of $P({\bf m}_1,{\bf m}_2)$. 
\end{itemize}
Therefore, we obtain the result by \eqref{eq:expansion of maximal vector} and \eqref{eq:expansion of maximal vector-2}.
\end{proof}

\begin{cor}\label{cor:lattice of KR level 2}
We have $\mc{L}_\ve^{(2)} = \mc{L}_\ve^{\ot 2}\cap \W^{(2)}_\ve$.
\end{cor}
\begin{proof} It is clear that $\mc{L}_\ve^{(2)} \subset \mc{L}_\ve^{\ot 2}\cap \W^{(2)}_\ve$ by Proposition \ref{prop:crystal base of W^{(2)}_+}. 
Conversely, suppose that $v\in \mc{L}_\ve^{\ot 2}\cap \W^{(2)}_\ve$ is given. By \eqref{eq:basis of level 2 KR}, we have
\begin{equation}\label{eq:expansion of v}
v= \sum_{T} c_T {\bf v}_T,
\end{equation} 
for some $c_T\in \Q(q)$. We may assume that all the shape of $T$ in \eqref{eq:expansion of v} is the same.
Fix $T$ with $c_T\neq 0$. Let $|{\bf m}_1\rangle \ot |{\bf m}_2\rangle$ be such that 
$|{\bf m}_1\rangle \ot |{\bf m}_2\rangle$ appears in \eqref{eq:expansion of v} with non-zero coefficient, and $P({\bf m}_1,{\bf m}_2)=T$. By Lemma \ref{eq:expansion of basis vector}, the coefficient of $|{\bf m}_1\rangle \ot |{\bf m}_2\rangle$ is $c_T$. Hence $c_T\in A_0$, and $v\in \mc{L}_\ve^{(2)}$. 
\end{proof}

\noindent{\em Proof of Theorem \ref{thm:main-1}.} 
Let $\W_\ve^{\ot 2}=\W^{(2)}_\ve \oplus W$, where $W$ is the complement of $\W^{(2)}_\ve $ in $\W_\ve^{\ot 2}$ as a $U_q(A_{n-1})$-module since it is completely reducible.
By Corollary \ref{cor:lattice of KR level 2}, we have
\begin{equation}\label{eq:decomp of L2}
\mc{L}_\ve^{\ot 2} = \mc{L}_\ve^{(2)} \oplus \mc{M}^{(2)},
\end{equation}
where $\mc{M}^{(2)}=\mc{L}_\ve^{\ot 2}\cap W$ is the  crystal lattice of $W$ as a $U_q(A_{n-1})$-module. 
Then  we have 
\begin{equation}\label{eq:image of L^(2)}
\check{R}_2 (\mc{L}_\ve^{\ot 2})\subset \mc{L}_\ve^{(2)},\quad \check{R}_2|_{q=0}(\mc{B}_\ve^{\ot 2})\subset \mc{B}_\ve^{(2)}. 
\end{equation}
More generally, by \eqref{eq:spectral decomp} and \eqref{eq:decomp of L2}, we have for $a\in \Z_{>0}$
\begin{equation}\label{eq:invariance of L}
\check{R}(q^{-2a})(\mc{L}_\ve^{\ot 2})\subset \mc{L}_\ve^{\ot 2}.
\end{equation}

For each $1\leq i\leq s-1$, we have
\begin{equation*}
\check{R}_s = \check{R}_{s_i}(\cdots, \underbrace{q^{s-2i-1}}_{i},\underbrace{q^{s-2i+1}}_{i+1},\cdots) \check{R}_{w_0s_i}(q^{1-s},\cdots, q^{s-1}).
\end{equation*}
We have $\check{R}_{w_0s_i}(q^{1-s},\cdots, q^{s-1})(\mc{L}_\ve^{\ot s})\subset \mc{L}_\ve^{\ot s}$ by \eqref{eq:invariance of L}, and hence by \eqref{eq:image of L^(2)}
\begin{equation*}\label{eq:image of R_s}
\begin{split}
&\check{R}_s(\mc{L}_\ve^{\ot s}) \subset \mc{L}_\ve^{\otimes i-1}\otimes \mc{L}_\ve^{(2)} \otimes \mc{L}_\ve^{\otimes s-i-1},\\
&\check{R}_s(\mc{B}_\ve^{\ot s}) \subset \mc{B}_\ve^{\otimes i-1}\otimes \mc{B}_\ve^{(2)} \otimes \mc{B}_\ve^{\otimes s-i-1}.
\end{split}
\end{equation*}
Therefore $\check{R}_s(\mc{B}_\ve^{\ot s})$ is spanned by $\mc{B}_\ve^{(s)}$, where
\begin{equation*}
\mc{B}_\ve^{(s)}=\left\{\,| \bm_1 \rangle\ot \dots\ot | \bm_s \rangle \!\!\!\pmod{q\mc{L}_\ve^{\ot s}}
\,\Big|\,|\bm_j\rangle \ot |\bm_{j+1}\rangle \in \mc{B}_\ve^{(2)}\ (1\leq j\leq s-2) \,\right\}.
\end{equation*}
By Proposition \ref{prop:crystal of level 2}, the set
\begin{equation*}
\left\{\,T({\bf m}_1,\dots,{\bf m}_s)\,\Big|\,| \bm_1 \rangle\ot \dots\ot | \bm_s\rangle \in \mc{B}_\ve^{(s)}\,\right\}
\end{equation*}
is equal to the set of semistandard tableau of shape $\la^\pi$ where $\la=(|{\bf m}_s|\geq \dots \geq |{\bf m}_1|)$.
Hence
\begin{equation}\label{eq:character of Ws}
{\rm ch}\W_\ve^{(s)} =\sum_{\substack{\la \in \cP_\ve \\ \ell(\la)\leq s}}s_{\la}(x_1,\dots,x_n).
\end{equation}
Let $V_0^{(s)}$ be the $U_q(C_n)$-submodule of  $\W_\ve^{(s)}$ generated $|\varsigma(\ve)\be_n \rangle^{\ot s}$. The classical limit $\ov{V_0^{(s)}}$ of $V_0^{(s)}$ is a highest weight $U(C_n)$-module with highest weight \[\Lambda^{(s)}:=-s(\frac{1}{2}+\varsigma(\ve))\varpi_n.\] 

On the other hand, the ($U(A_{n-1})$-)character of the irreducible highest weight $U(C_n)$-module with highest weight $\Lambda^{(s)}$, say $V(\Lambda^{(s)})$, is also equal to \eqref{eq:character of Ws} by \cite[Theorem 6.1]{K18}.
Since $V(\Lambda^{(s)})$ is a quotient of $\ov{V_0^{(s)}}$, we conclude that
\begin{equation*}
{\rm ch}\W_\ve^{(s)}={\rm ch}V_0^{(s)}={\rm ch}\ov{V_0^{(s)}}={\rm ch}V(\Lambda^{(s)}).
\end{equation*}
In particular, $V_0^{(s)}$ is an irreducible $U_q(C_n)$-module and hence so is $\W_\ve^{(s)} = V_0^{(s)}$.
This completes the proof.
\qed

\subsection{Type $C^{(2)}(n+1)$}

Next, let us prove that $\W^{(s)}$ is an irreducible $U_q(osp_{1|2n})$-module and compute its character. The proof is similar to that of Theorem \ref{thm:main-1} for $U_q(C_n^{(1)})$. So we give a sketch of the proof and leave the details to the reader. 

As usual, we realize the weight lattice for $U_q(osp_{1|2n})$ in $\bigoplus_{i=1}^n\mathbb{R}\be_i$ equipped with the standard symmetric bilinear form such that $(\be_i,\be_j)=\delta_{ij}$.  
The simple roots $\alpha_i$ $(i\in I\setminus\{0\})$ are given by $\boldsymbol{\alpha}_i=\be_{i+1}-\be_i$ for $1\leq i\leq n-1$ and $\boldsymbol{\alpha}_n=-\be_n$, and $\varpi_n=-\hf(\be_1+\dots+\be_n)$.

We first consider $\W^{(2)}$.
By \eqref{eq:spectral decomp-2}, we have
\begin{equation}
\W^{(2)} = V_0= U_q(osp_{1|2n})\left(|{\bf 0}\rangle\ot |{\bf 0}\rangle\right),
\end{equation}
which is an irreducible representation of $U_q(osp_{1|2n})$ and hence of $U_q(C^{(1)}(n+1))$.
By similar arguments as in Proposition \ref{prop:ch of level 2}, we have the following.
\begin{prop}\label{prop:ch of W^2}
We have
\begin{equation*}
{\rm ch}\W^{(2)}={\rm ch}V_0 = \sum_{\substack{\la \in\cP \\ \ell(\la)\leq 2}} s_{\la}(x_1,\dots,x_n).
\end{equation*}
\end{prop} 

\begin{lem}
For $l\in \Z_{\ge 0}$, let ${\bf v}_l$ be the vector in \eqref{eq:maximal vector of type A}. Then ${\bf v}_l$ is a $U_q(A_{n-1})$-highest weight vector in $\W^{(2)}$, and ${\bf v}_l \equiv | l\be_n \rangle \ot | l\be_{n-1} \rangle \pmod{q^\hf\mc{L}^{\ot 2}}$.
\end{lem}
\begin{proof} 
Since the actions of $e_i, f_i, k_i^{\pm 1}$ for $1\le i\le n-1$ are the same as in the case of $C_n^{(1)}$, it follows from Lemma \ref{lem:v_l} that ${\bf v}_l$ is a $U_q(A_{n-1})$-highest weight vector. Note that
\begin{equation}
{\rm ch}\W^{\ot 2}=({\rm ch}\W)^{2} = \sum_{\ell(\la)\leq 2}a_\la s_{\la}(x_1,\dots,x_n),
\end{equation}
where $a_\la=\la_1-\la_2$.
Then we have ${\bf v}_l\in\W^{(2)}$ by the same argument as in Lemma \ref{lem:v_l}.  \end{proof}\vskip 2mm

We have an analogue of Lemma \ref{lem:operator E}, which also proves that ${\bf v}_l\in \W^{(2)}$.
\begin{lem}\label{lem:operator E-2}
Set $\mc{E}=e_{n-2}\cdots e_1e_0$, where it is understood as 
$e_0$ when $n=2$. Then for $l\ge0$ we have
\[
(\mc{E}e_{n-1}\mc{E}-\frac{1}{[2]}e_{n-1}\mc{E}^2){\bf v}_l
=(-1)^lq^{-5/2}\frac{(1+q)}{[2]}[l+1]^2{\bf v}_{l+1}.
\]
\end{lem}

\begin{lem}\label{lem:v_{l,m} at q=0-2}
For $l, m\in \Z_{\geq 0}$, let
\begin{equation*}
{\bf v}_{l,m} = q_n^{\frac{m(m+4l+3)}{2}}f_n^{(m)} {\bf v}_l.
\end{equation*}
Then ${\bf v}_{l,m}$ is a $U_q(A_{n-1})$-highest weight vector in $\W^{(2)}$, and 
\begin{equation*}
{\bf v}_{l,m} \equiv | l\be_n \rangle \ot | l\be_{n-1} +m\be_n\rangle \pmod{q^\hf\mc{L}^{\ot 2}}.
\end{equation*} 
\end{lem}
\begin{proof} 
Since $e_j$ for $1\leq j\leq n-1$ commutes with $f_n$, ${\bf v}_{l,m}$ is a $U_q(A_{n-1})$-highest weight vector in $\W_\e^{(2)}$.

For $0\leq c\leq l$, put $a=l-c$ and $b=c$.
Let
\begin{equation*}
u_1 = |c\be_{n-1}\rangle,\quad u_2=|(l-c)\be_{n-1}\rangle.
\end{equation*}
By \eqref{eq:comult-2}, we have
{\allowdisplaybreaks
\begin{align*}
&\Delta(f_n^{(m)})(\tf_n^a u_1 \ot \tf_n^b u_2) \\
&=\sum_{k=0}^m\sigma^{k}q_n^{-k(m-k)}f_n^{(m-k)}k_n^k(\tf_n^a u_1)\ot f_n^{(k)}(\tf_n^b u_2)\\
&=\sum_{k=0}^m\sigma^{k}q_n^{-k(m-k)-(1+2a)k}f_n^{(m-k)}\tf_n^{a}u_1\ot f_n^{(k)}\tf_n^{b}u_2\\
&=\sum_{k=0}^m\sigma^{k}q_n^{-k(m-k) -(1+2a)k+\frac{a(a+1)}{2}+\frac{b(b+1)}{2}}f_n^{(m-k)}f_n^{(a)}u_1 \ot f_n^{(k)}f_n^{(b)}u_2 \\
&=\sum_{k=0}^m\sigma^{k}q_n^{-k(m-k) -(1+2a)k+\frac{a(a+1)}{2}+\frac{b(b+1)}{2}}
\begin{bmatrix} m-k+a \\ a \end{bmatrix}_n\begin{bmatrix} k+b \\ b \end{bmatrix}_n
f_n^{(m-k+a)}u_1 \ot f_n^{(k+b)}u_2 \\
&=\sum_{k=0}^m\sigma^{k}q_n^{-k(m-k) -(1+2a)k+\frac{a(a+1)}{2}+\frac{b(b+1)}{2}-\frac{(m-k+a)(m-k+a+1)}{2}-\frac{(k+b)(k+b+1)}{2}} \\ & \hskip 5cm \begin{bmatrix} m-k+a \\ a \end{bmatrix}_n\begin{bmatrix} k+b \\ b \end{bmatrix}_n
\tf_n^{m-k+a}u_1 \ot \tf_n^{k+b}u_2 \\
&=\sum_{k=0}^m\sigma^{k}f_{a,b}(q)\tf_n^{m-k+a}u_1 \ot \tf_n^{k+b}u_2.
\end{align*}}
Multiplying $q_n^{\frac{m(m+4l+3)}{2}}$ on both sides, it is straightforward to see that \[q_n^{\frac{m(m+4l+3)}{2}}f_{a,b}(q)\in q_n^{(2c+1)(m-k)}(1+q^{\hf}A_0).\]
This implies that ${\bf v}_{l,m} \equiv | l\be_n \rangle \ot | l\be_{n-1} +m\be_n\rangle \pmod{q^\hf\mc{L}^{\ot 2}}$.
\end{proof}

Now we define the pair $(\mc{L}^{(2)},\mc{B}^{(2)})$ as follows:
\begin{equation*}
\begin{split}
&\mc{L}^{(2)}=
\sum_{l_1, l_2\in \Z_{\geq 0}}
\sum_{\substack{r\geq 0\\ 1\leq i_1,\ldots,i_r\leq n-1}}
A_0 \tf_{i_1}\ldots \tf_{i_r} {\bf v}_{l_1,l_2},\\
&\mc{B}^{(2)}=
\left\{\,\tf_{i_1}\ldots \tf_{i_r} {\bf v}_{l_1,l_2} \!\!\!\pmod{q^\hf\mc{L}^{(2)}}\,\Big|\, l_1, l_2\in \Z_{\geq 0},\ r\geq 0,\ 1\leq i_1,\ldots,i_r\leq n-1\,\right\}\setminus\{0\}.
\end{split}
\end{equation*}

\begin{prop}\label{prop:crystal base of W^{(2)}}
We have
\begin{itemize}
\item[(1)] $\mc{L}^{(2)}\subset \mc{L}^{\ot 2}$ and $\mc{B}^{(2)}\subset \mc{B}^{\ot 2}$,

\item[(2)] $(\mc{L}^{(2)},\mc{B}^{(2)})$ is a $U_q(A_{n-1})$-crystal base of $\W^{(2)}$, where
\begin{equation*}
\mc{B}^{(2)} =\left\{\,|{\bf m}_1\rangle \ot |{\bf m}_2\rangle\!\!\!\pmod{q^\hf\mc{L}^{(2)}} \,\Big|\,|{\bf m}_1|\leq |{\bf m}_2|,\ \text{\rm $T({\bf m}_1,{\bf m}_2)$ is semistandard}\,\right\}.
\end{equation*}
\end{itemize}
\end{prop}
\begin{proof}
It follows from  the same arguments as in Propositions \ref{prop:crystal base of W^{(2)}_+} and \ref{prop:crystal of level 2}.
\end{proof}

\begin{cor}\label{cor:lattice of KR level 2-2}
We have $\mc{L}^{(2)} = \mc{L}^{\ot 2}\cap \W^{(2)}$.
\end{cor}
\begin{proof} By Proposition \ref{prop:crystal base of W^{(2)}}, one can check that
Lemma \ref{eq:expansion of basis vector} also holds for $\W^{(2)}$, which implies $\mc{L}^{(2)} = \mc{L}^{\ot 2}\cap \W^{(2)}$.
\end{proof}

\begin{thm}\label{thm:main-2}
For $s\geq 2$, $\W^{(s)}$ is irreducible as a $U_q(osp_{1|2n})$-module, and its character is given by
\begin{equation*}
{\rm ch}\W^{(s)} =\sum_{\substack{\la \in \cP \\ \ell(\la)\leq s}}s_{\la}(x_1,\dots,x_n).
\end{equation*}
\end{thm}
\begin{proof}
We may apply the same arguments as in Theorem \ref{thm:main-1} and the result in \cite[Theorem 6.1]{K18} by using Proposition \ref{prop:crystal base of W^{(2)}} and Corollary \ref{cor:lattice of KR level 2-2}.
\end{proof}

\begin{cor}
The character of $\W_\ve^{(s)}$ has a stable limit for $s\geq n$ as follows: 
\begin{equation*}
{\rm ch}\W^{(s)} 
=\sum_{\substack{\la \in \cP \\ \ell(\la)\leq n}}s_{\la}(x_1,\dots,x_n)
=\frac{1}{\prod_{1\leq i\leq n}(1-x_i)\prod_{1\leq i< j\leq n}(1-x_ix_j)}.
\end{equation*}
\end{cor}

\subsection{Type $B^{(1)}(0,n)$}
For $\la\in \cP_+$ with $\ell(\la)\leq \min\{n,s/2\}$, we put
\begin{equation*}
\La^{(s)}_\la = -s\varpi_n + \sum_{i=1}^n\la_i\be_{n-i+1}.
\end{equation*}
Let $V(\La^{(s)}_\la)$ be the irreducible highest weight $U(osp_{1|2n})$-module with highest weight $\La^{(s)}_\la$. Note that $\La^{(2)}_{(l)}$ is the weight of the maximal vector $v_l$ and $V(\La^{(2)}_{(l)})=V_l$ for $l\geq 0$. Generalizing the decomposition of $\W^{(2)}$ into $U_q(osp_{1|2n})$-modules, we have the following conjecture on $\W^{(s)}$.

\begin{conj}\label{conj:KR}
For $s\geq 2$, the character of $\W^{(s)}$ is given by
\begin{equation*}
{\rm ch}\W^{(s)} = \sum_{\substack{\la\in \cP_+ \\ \ell(\la)\leq \min\{n,s/2\}}} 
{\rm ch}V(\La^{(s)}_\la).
\end{equation*}
\end{conj}

\begin{rem}{\rm
The family of infinite-dimensional $U(osp_{1|2n})$-modules $V(\La^{(s)}_{\la})$ have been introduced in \cite{CKW} in connection with Howe duality. They are unitarizable and form a semisimple tensor category. The Weyl-Kac type character formula for $V(\La^{(s)}_{\la})$ can be found in \cite[Theorem 6.13]{CKW}, while a combinatorial formula is given in \cite[Corollary 6.6]{K18}.}
\end{rem}

\begin{cor}
For $s\geq 2n$, we have
\begin{equation*}
\begin{split}
{\rm ch}\W^{(s)} &= \frac{\sum_{\la\in \cP_+}s_\la(x_1,\dots,x_n)}{\prod_{1\leq i\leq n}(1-x_i)\prod_{1\leq i<j\leq n}(1-x_ix_j)}\\
&= \frac{1}{\prod_{1\leq i\leq n}(1-x_i)(1-x_i^2)\prod_{1\leq i<j\leq n}(1-x_ix_j)^2}.
\end{split}
\end{equation*}
\end{cor}
\begin{proof}
The first equality follows from the fact \cite[Corollary 6.6]{K18} that 
if $\la\in \cP_+$ with $\ell(\la)\leq n$, then
\begin{equation*}
{\rm ch}V(\La^{(s)}_\la) = \frac{s_\la(x_1,\dots,x_n)}{\prod_{1\leq i\leq n}(1-x_i)\prod_{1\leq i<j\leq n}(1-x_ix_j)}.
\end{equation*}
The second one follows from the well-known Littlewood identity.
\end{proof}

\appendix

\section{Twistor}\label{app:twistor}

In this appendix, we prove Propositions \ref{prop:q-osc for C^{(2)}(n+1)} and \ref{prop:q-osc for B^{(1)}(0,n)}. We first review the twistor introduced in \cite{CFLW} that relate quantum groups to quantum supergroups. Then we use it to relate the $q$-oscillator representation
of $U_q(D_{n+1}^{(2)})$ in \cite{KO} to a representation of $U_q(C^{(2)}(n+1))$. An 
advantage to do so is that in the latter case we can take a classical limit $q\to1$.
We also obtain a representation of $U_q(B^{(1)}(0,n))$ from the $q$-oscillator
representation of $U_q(A_{2n}^{(2)\dagger})$, where $A_{2n}^{(2)\dagger}$ is the
same Dynkin diagram as $A_{2n}^{(2)}$ in \cite{Kac} but the labeling of nodes are
opposite.

\subsection{The twistor of the covering quantum group} \label{subsec:twistor}

We review the covering quantum group and the twistor map introduced in
\cite{CFLW}. Our notations for a Cartan datum is closer to Kac's book \cite{Kac}.
Let $I$ be the index set of the Dynkin diagram, $\{\alpha_i\}_{i\in I}$ the set of
simple roots, $(a_{ij})_{i,j\in I}$ the Cartan matrix. The symmetric bilinear form 
$(\cdot,\cdot)$ on the weight lattice is normalized so that it satisfies
$d_i=(\alpha_i,\alpha_i)/2\in\Z$ for any $i\in I$. It is also assumed that
$a_{ij}\in2\Z$ if $d_i\equiv1$ (mod 2) and $j\in I$. The parity function $p(i)$ 
taking values in $\{0,1\}$ is consistent with $d_i$, namely, $p(i)\equiv d_i$ (mod 2).
We set $q_i=q^{d_i},\pi_i=\pi^{d_i}$.

Let $q,\pi$ be indeterminates and $\ib=\sqrt{-1}$. For a ring $R$ with 1, we 
set $R^\pi=R[\pi]/(\pi^2-1)$.
The covering quantum group  $\mathbf{U}$ associated to a Cartan datum 
is the $\Q^\pi(q,\ib)$-algebra with generators $E_i,F_i,K_i^{\pm1},J_i^{\pm1}$ for 
$i\in I$ subject to the following relations:
\begin{align*}
&J_iJ_j=J_jJ_i,\quad K_iK_j=K_jK_i,\quad J_iK_j=K_jJ_i,\\
&K_iK_i^{-1}=K_i^{-1}K_i=J_iJ_i^{-1}=J_i^{-1}J_i=J_i^2=1,\\
&J_iE_j=\pi^{a_{ij}}E_jJ_i,\quad J_iF_j=\pi^{-a_{ij}}F_jJ_i,\\
&K_iE_j=q^{a_{ij}}E_jK_i,\quad K_iF_j=q^{-a_{ij}}F_jK_i,\\
&E_iF_j-\pi^{p(i)p(j)}F_jE_i=\delta_{ij}\frac{J_iK_i-K_i^{-1}}{\pi_iq_i-q_i^{-1}},\\
&\sum_{l=0}^{1-a_{ij}}(-1)^l\pi^{l(l-1)p(i)/2+lp(i)p(j)}
{1-a_{ij}\brack l}_{q_i,\pi_i}E_i^{1-a_{ij}-l}E_jE_i^l=0\quad(i\ne j),\\
&\sum_{l=0}^{1-a_{ij}}(-1)^l\pi^{l(l-1)p(i)/2+lp(i)p(j)}
{1-a_{ij}\brack l}_{q_i,\pi_i}F_i^{1-a_{ij}-l}F_jF_i^l=0\quad(i\ne j).
\end{align*}

\begin{rem}{\rm 
We changed the notations from \cite{CFLW}. We replaced $v$ with $q$, $\mathbf{t}$
with $\ib$, and $J_{d_ii}, K_{d_ii}, T_{d_ii}$ with $J_i, K_i, T_i$.}
\end{rem}

We extend $\bf U$ by 
introducing the generators $T_i,\Upsilon_i$ for $i\in I$. They commute with each other and
with $J_i,K_i$. They also have the commutation relations with $E_i,F_i$ as 
\[
T_iE_j=\ib^{d_ia_{ij}}E_jT_i,\quad T_iF_j=\ib^{-d_ia_{ij}}F_jT_i,\quad 
\Upsilon_iE_j=\ib^{\phi_{ij}}E_j\Upsilon_i,\quad \Upsilon_iF_j=\ib^{-\phi_{ij}}F_j\Upsilon_i,
\]
where 
\[
\phi_{ij}=\begin{cases}
d_ia_{ij}&\text{if }i>j,\\
d_i&\text{if }i=j,\\
-2p(i)p(j)&\text{if }i<j.
\end{cases}
\]
We denote this extended algebra by $\widehat{\mathbf{U}}$.

\begin{thm}[\cite{CFLW}]
There is a $\Q(\ib)$-algebra automorphism $\widehat{\Psi}$ on $\widehat{\mathbf{U}}$
such that
\begin{alignat*}{3}
&E_i\mapsto \ib^{-d_i}\Upsilon_i^{-1}T_iE_i,&\quad
&F_i\mapsto F_i\Upsilon_i,&\quad
&K_i\mapsto T_i^{-1}K_i,\\
&J_i\mapsto T_i^2J_i,&\quad
&T_i\mapsto T_i,&\quad
&\Upsilon_i\mapsto \Upsilon_i,\\
&q\mapsto \ib^{-1}q,&\quad
&\pi\mapsto -\pi.
\end{alignat*}
\end{thm}

\subsection{Image of the twistor $\widehat{\Psi}$}\label{app:proof of well-definedness}

We apply the twistor $\widehat{\Psi}$ given in the previous subsection for the Cartan
datum corresponding to $B_n$, namely, $I=\{1,2,\ldots,n\}$ and the Cartan matrix is 
given by 
\[
(a_{ij})=
\begin{pmatrix}
2&-1&\\
-1&2&-1\\
&-1& \\
&&&\ddots\\
&&&&&-1\\
&&&&-1&2&-1\\
&&&&&-2&2
\end{pmatrix}
\]
Through it, we are to regard  the $q$-oscillator representation 
$\W=\bigoplus_{\mb}\Q(q^{\frac12})\ket{\mb}$ of $U_q(B_n)$, the subalgebra of $U_q(D_{n+1}^{(2)})$ generated by $e_i,f_i,k_i$ for $i\in I\setminus \{0\}$, given in \cite[Proposition 1]{KO} as a representation of $U_q(osp_{1|2n})$. 
Although we normalized the symmetric bilinear form on the weight lattice so that
$(\alpha_i,\alpha_i)\in 2\Z$ for any $i\in I$ in the previous subsection, 
we renormalize it so that $(\alpha_n,\alpha_n)=1$ to adjust it to the notations
in \cite{KO}. The generators $T_i,\Upsilon_i$ are represented on $\W$ as
\[
T_i\ket{\mb}=\begin{cases}
\ib^{2(m_{i+1}-m_i)}\ket{\mb}&(1\le i<n)\\
\ib^{-2m_n}\ket{\mb}&(i=n)
\end{cases},\quad
\Upsilon_i\ket{\mb}=\begin{cases}
\ib^{-2m_i}\ket{\mb}&(1\le i<n)\\
\ib^{|\mb|-2m_n}\ket{\mb}&(i=n)
\end{cases}.
\]

Let $u_i$ ($i\in I$, $u=e,f,k$) be the generators of $U_q(B_n)$ ($\pi=1$) 
and $\bar{u}_i=\widehat{\Psi}(u_i)$
be the image ($\pi=-1$) of the twistor $\widehat{\Psi}$. Then $\bar{u}_i$
satisfy the relations for $U_{\bar{q}}(osp_{1|2n})$
where $\bar{q}^{\frac12}=\ib^{-1}q^{\frac12}$. On the 
space $\W$, they act as follows:
\begin{align*}
\bar{e}_i\ket{\mb}&=\ib^{2m_{i+1}}[m_i]\ket{\mb-\eb_i+\eb_{i+1}},\\
\bar{f}_i\ket{\mb}&=\ib^{-2m_i}[m_{i+1}]\ket{\mb+\eb_i-\eb_{i+1}},\\
\bar{k}_i\ket{\mb}&=\ib^{2m_i-2m_{i+1}}q^{-m_i+m_{i+1}}\ket{\mb},\\
\bar{e}_n\ket{\mb}&=\kappa\, \ib^{1-|\mb|}[m_n]\ket{\mb-\eb_n},\\
\bar{f}_n\ket{\mb}&=\ib^{|\mb|-2m_n}\ket{\mb+\eb_n},\\
\bar{k}_n\ket{\mb}&=\ib^{2m_n+1}q^{-m_n-\frac12}\ket{\mb},
\end{align*}
where $1\le i<n$, $\kappa=(q+1)/(q-1)$. 

By introducing the actions of $\bar{e}_0,\bar{f}_0,\bar{k}_0$, we want to make $\W$ 
a representation of the quantum affine superalgebra associated to $C^{(2)}(n+1)$ or $B^{(1)}(0,n)$. For the former, we set
\begin{align*}
\bar{e}_0\ket{\mb}&=x\,\ib^{2m_1-|\mb|}\ket{\mb+\eb_1},\\
\bar{f}_0\ket{\mb}&=x^{-1}\kappa\,\ib^{|\mb|+1}[m_1]\ket{\mb-\eb_1},\\
\bar{k}_0\ket{\mb}&=\ib^{-2m_1-1}q^{m_1+\frac12}\ket{\mb},
\end{align*}
and for the latter 
\begin{align*}
\bar{e}_0\ket{\mb}&=x(-1)^{|\mb|}\ket{\mb+2\eb_1},\\
\bar{f}_0\ket{\mb}&=x^{-1}(-1)^{|\mb|}\frac{[m_1][m_1-1]}{[2]^2}\ket{\mb-2\eb_1},\\
\bar{k}_0\ket{\mb}&=-q^{2m_1+1}\ket{\mb},
\end{align*}
where $x$ is the so-called spectral parameter. We also note that the quantum 
parameter is still $\bar{q}^{\frac12}=\ib^{-1}q^{\frac12}$.

To obtain the representation for the quantum parameter $q$, we need to 
we switch $q^{\frac12}$ to $\ib q^{\frac12}$ ($\bar{q}^{\frac12}$ to $q^{\frac12}$).
Also, the relations in Section \ref{subsec:twistor} and those in Section 
\ref{subsec:quantum affine superalgebra} are
different. For the node $i$ that is signified as $\bullet$ in the Dynkin diagram,
there is a relation
\[
e_if_i+f_ie_i=\frac{k_i-k_i^{-1}}{q^{\frac12}-q^{-\frac12}}
\]
in Section \ref{subsec:quantum affine superalgebra} rather than 
\[
e_if_i+f_ie_i=\frac{k_i-k_i^{-1}}{-q^{\frac12}-q^{-\frac12}}
\]
in Section \ref{subsec:twistor}. The former relation is realized by deleting $\kappa$
from the action of $\bar{e}_i$ or $\bar{f}_i$ in the formulas of the $q$-oscillator
representation above. By doing so, we obtain
{\allowdisplaybreaks \begin{align*}
\bar{e}_0\ket{\mb}&=
\begin{cases}
x\,\ib^{2m_1-|\mb|}\ket{\mb+\eb_1}&\text{for }U_q(C^{(2)}(n+1))\\
x(-1)^{|\mb|}\ket{\mb+2\eb_1}&\text{for }U_q(B^{(1)}(1,n))
\end{cases},\\
\bar{f}_0\ket{\mb}&=
\begin{cases}
x^{-1}\ib^{|\mb|+2m_1+1}[m_1]\ket{\mb-\eb_1}&\text{for }U_q(C^{(2)}(n+1))\\
x^{-1}(-1)^{|\mb|+1}\frac{[m_1][m_1-1]}{[2]^2}\ket{\mb-2\eb_1}&\text{for }U_q(B^{(1)}(0,n))
\end{cases},\\
\bar{k}_0\ket{\mb}&=
\begin{cases}
q^{m_1+\frac12}\ket{\mb}&\text{for }U_q(C^{(2)}(n+1))\\
q^{2m_1+1}\ket{\mb}&\text{for }U_q(B^{(1)}(0,n))
\end{cases},\\
\bar{e}_i\ket{\mb}&=(-1)^{-m_i+m_{i+1}+1}[m_i]\ket{\mb-\eb_i+\eb_{i+1}},\\
\bar{f}_i\ket{\mb}&=(-1)^{-m_i+m_{i+1}+1}[m_{i+1}]\ket{\mb+\eb_i-\eb_{i+1}},\\
\bar{k}_i\ket{\mb}&=q^{-m_i+m_{i+1}}\ket{\mb},\\
\bar{e}_n\ket{\mb}&=\ib^{1-|\mb|+2m_n}[m_n]\ket{\mb-\eb_n},\\
\bar{f}_n\ket{\mb}&=\ib^{|\mb|-2m_n}\ket{\mb+\eb_n},\\
\bar{k}_n\ket{\mb}&=q^{-m_n-\frac12}\ket{\mb},
\end{align*}}
for $1\le i\le n-1$.

Finally, to obtain the actions of $U_q(C^{(2)}(n+1))$ and $U_q(B^{(1)}(0,n))$) in 
Propositions \ref{prop:q-osc for C^{(2)}(n+1)} and \ref{prop:q-osc for B^{(1)}(0,n)}, respectively,
we perform the basis change $\ket{\mb}$ to
$$\ib^{s(\mb)}q^{-|\mb|/2}\prod_{j=1}^n[m_j]!\ket{\mb},$$ where 
$s(\mb)=-|\mb|(|\mb|+1)/2-\sum_jm_j^2$. Next we apply the algebra automorphism 
sending $e_n\mapsto-e_n,f_n\mapsto-f_n$ and the other generators fixed. For 
$U_q(C^{(2)}(n+1))^\sigma$, we also apply $e_0\mapsto\sigma e_0$, 
$f_0\mapsto f_0\sigma$. Accordingly, the coproduct also changes.
For $U_q(B^{(1)}(0,n))$, we alternatively apply $e_0\mapsto\ib[2]e_0,f_0\mapsto
\frac1{\ib[2]}f_0$. This completes the proof.

\section{Quantum $R$ matrix for $U_q(A_{2n}^{(2)\dagger})$} \label{app:Adagger}

In this appendix, we consider the quantum $R$ matrix for the $q$-oscillator 
representation of $U_q(A_{2n}^{(2)\dagger})$ where $A_{2n}^{(2)\dagger}$ is the
Dynkin diagram whose nodes have the opposite labelings to $A^{(2)}_{2n}$. 
This will be used in Appendix \ref{app:R matrix for super} to derive the quantum $R$ matrix for $U_q(B^{(1)}(0,n))$.

\subsection{$q$-oscillator representation for $U_q(A^{(2)\dagger}_{2n})$}

By $A^{(2)\dagger}_{2n}$ we denote the following Dynkin diagram.

\[
\vcenter{\xymatrix@R=1ex{
*{\circ}<3pt> \ar@{=}[r] |-{\scalebox{2}{\object@{>}}}_<{0} 
&*{\circ}<3pt> \ar@{-}[r]_<{1} 
& {} \ar@{.}[r]&{}  \ar@{-}[r]_>{\,\,\,\,n-1} &
*{\circ}<3pt> \ar@{=}[r] |-{\scalebox{2}{\object@{>}}}
& *{\circ}<3pt>\ar@{}_<{n}}}
\]

\noindent
Although we did not deal with the $q$-oscillator representation for 
$U_q(A^{(2)\dagger}_{2n})$ in \cite{KO}, it is easy to guess from other cases given there.
On the space $\W$, the actions are given as follows.
{\allowdisplaybreaks
\begin{align*}
e_0\ket{\mb}&=x\ket{\mb+2\eb_1},\\
f_0\ket{\mb}&=x^{-1}\frac{[m_1][m_1-1]}{[2]^2}\ket{\mb-2\eb_1},\\
k_0\ket{\mb}&=-q^{2m_1+1}\ket{\mb},\\
e_i\ket{\mb}&=[m_i]\ket{\mb-\eb_i+\eb_{i+1}},\\
f_i\ket{\mb}&=[m_{i+1}]\ket{\mb+\eb_i-\eb_{i+1}},\\
k_i\ket{\mb}&=q^{-2m_i+2m_{i+1}}\ket{\mb},\\
e_n\ket{\mb}&=\ib\kappa[m_n]\ket{\mb-\eb_n},\\
f_n\ket{\mb}&=\ket{\mb+\eb_n},\\
k_n\ket{\mb}&=\ib q^{-m_n-1/2}\ket{\mb},
\end{align*}}
where $0<i<n$ and $\kappa=(q+1)/(q-1)$. Denote this representation map by $\pi_x$.

Recall that the $U_q(B_n)$-highest weight vectors $\{v_l\mid l\in\Z_{\ge0}\}$ are calculated in
\cite[Proposition 4]{KO}. We take the coproduct \eqref{eq:coproduct for quantum covering}
with $\pi=1$.

\begin{lem}\label{lem:App B}
For $x,y\in\Q(q)$ we have 
\begin{enumerate}
\item $(\pi_x\ot\pi_y)\Delta(f_0f_1^{(2)}\cdots f_{n-1}^{(2)})v_l=
-\frac{[l][l-1]}{[2]^2}(q^{2l-2}x^{-1}+q^{-1}y^{-1})v_{l-2}\quad(l\ge2)$,
\item $(\pi_x\ot\pi_y)\Delta(e_ne_{n-1}^{(2)}\cdots e_1^{(2)}e_0)v_0=
\frac{\ib\kappa [2]}{1-q}((y+qx)v_1-q(y+x)\Delta(f_n)v_0)$.
\end{enumerate}
\end{lem}

Define the quantum $R$ matrix $\check{R}_{KO}(z,q)$ for $U_q(A^{(2)\dagger}_{2n})$ as in Proposition \ref{prop:new R matrix}. The existence of such $\check{R}_{KO}(z,q)$ is essentially
given in \cite[Theorem 13]{KO}. Namely, although $A^{(2)\dagger}_{2n}$ is not
listed there, the corresponding gauge transformed quantum $R$ matrix is 
$S^{2,1}(z)$ and the proof has been done as the cases (i),(iv) and (v). By using Lemma \ref{lem:App B}, we have the following.

\begin{prop} \label{prop:spectral decomp-app}
We have the following spectral decomposition
\begin{equation*}
\check{R}_{KO}(z)=\sum_{l\in 2\Z_{+}}
\prod_{j=1}^{l/2}\frac{z+q^{4j-1}}{1+q^{4j-1}z}P_l
+\sum_{l\in 1+2\Z_{+}}\prod_{j=0}^{(l-1)/2}\frac{z+q^{4j+1}}{1+q^{4j+1}z}P_l,
\end{equation*}
where $P_l$ is the projector on the subspace generated by the
$U_q(B_n)$-highest weight vector $v_l\;(l\ge0)$.
\end{prop}

\section{Quantum $R$ matrix for $U_q(C^{(2)}(n+1))$ and $U_q(B^{(1)}(0,n))$}\label{app:R matrix for super}

In this appendix, we compare the quantum $R$ matrix for the $q$-oscillator
representation for $U_q(C^{(2)}(n+1))$ with the one for $U_q(D_{n+1}^{(2)})$ given in 
\cite{KO}. We also compare the quantum $R$ matrix for $U_q(B^{(1)}(0,n))$ with the one for $U_q(A^{(2)\dagger}_{2n})$ in Appendix \ref{app:Adagger}. We keep the notations in Appendix \ref{app:twistor}.

\subsection{Gauge transformation}

We take the following coproduct
\begin{equation}\label{eq:coproduct for quantum covering}
\begin{split}
&\Delta(k_i)=k_i\ot k_i,\\
&\Delta(e_i)=1\ot e_i+e_i\ot \sigma^{\frac{1-\pi}2 p(i)}k_i,\\
&\Delta(f_i)=f_i\ot\sigma^{\frac{1-\pi}2 p(i)}+k_i^{-1}\ot f_i,\\
\end{split}
\end{equation}
for $i\in I$, where $\sigma$ satisfies \eqref{eq:relation for sigma}.
We also take the same coproduct \eqref{eq:coproduct for quantum covering} for $\ov{u}_i$. 
Let $\Gamma$ be an operator acting on $\W^{\ot2}$ by
\begin{equation}
\Gamma\ket{\mb}\otimes\ket{\mb'}
=\ib^{\sum_{k,l}\varphi_{kl}m_km'_l}\ket{\mb}\otimes\ket{\mb'},
\end{equation}
for $\bm=(m_1,\dots,m_n)$ and $\bm'=(m'_1,\dots,m'_n)$.
Here we have the constraint $\varphi_{kl}+\varphi_{lk}=0$. 
Then by \cite{R} (see also \cite{OY}), 
\[\Delta^\Gamma(u)=\Gamma^{-1}\Delta(u)\Gamma\] 
gives another coproduct of $U_q(B_n)$ acting on $\W^{\ot 2}$. 
Take $\varphi_{kl}$ to be 1 for $k<l$.
We also set 
\begin{equation}
K\ket{\mb}=\ib^{c(\mb)}\ket{\mb},
\end{equation}
where
\[
c(\mb)=-\frac12\sum_km_k^2+\sum_k\left(k-n-\frac12\right)m_k.
\]
Set 
\begin{align*}
\gamma_i(\mb)&=\begin{cases}
-|\mb|+m_1&(i=0\text{ and for }U_q(C^{(2)}(n+1)))\\
-2|\mb|+2m_1&(i=0\text{ and for }U_q(B^{(1)}(0,n)))\\
m_i+m_{i+1}&(0<i<n)\\
-|\mb|+m_n&(i=n)
\end{cases},\\
\beta_i(\mb)&=\begin{cases}
m_1+n&(i=0\text{ and }U_q(C^{(2)}(n+1)))\\
2m_1+2n+1&(i=0\text{ and }U_q(B^{(1)}(0,n)))\\
-m_i+m_{i+1}&(0<i<n)\\
-m_n&(i=n)
\end{cases}.
\end{align*}
Let $\boldsymbol{\alpha}_0=\eb_1\text{ for }U_q(C^{(2)}(n+1))),2\eb_1\text{ for }
U_q(B^{(1)}(0,n)))$,
$\boldsymbol{\alpha}_i=-\eb_i+\eb_{i+1}$ $(0<i<n)$, and $\boldsymbol{\alpha}_n=-\eb_n$. 

\begin{lem}\label{app:formula-1}
The following formulas hold for $\bm$, $\bm'$, and $i\in I$:
\begin{enumerate}
\item $\Gamma^{-1}(1\ot e_i)\Gamma\, \ket{\mb}\ot \ket{\mb'}
=\ib^{-\gamma_i(\mb)}\ket{\mb}\ot e_i\ket{\mb'}$,
\item $\Gamma^{-1}(e_i\ot 1)\Gamma\,\ket{\mb}\ot\ket{\mb'}=\ib^{\gamma_i(\mb')}e_i\ket{\mb}\ot\ket{\mb'}$,
\item $\Gamma^{-1}(1\ot f_i)\Gamma\,\ket{\mb}\ot\ket{\mb'}
=\ib^{\gamma_i(\mb-\boldsymbol{\alpha}_i)}\ket{\mb}\ot f_i\ket{\mb'}$,
\item $\Gamma^{-1}(f_i\ot 1)\Gamma\, \ket{\mb}\ot\ket{\mb'}=\ib^{-\gamma_i(\mb'-\boldsymbol{\alpha}_i)} f_i\ket{\mb}\ot\ket{\mb'}$.
\end{enumerate}
\end{lem}

\begin{lem}\label{app:formula-2}
The following formulas hold for $\bm$ and $i\in I$:
\begin{enumerate}
\item $K^{-1}e_iK\ket{\mb}=\ib^{\beta_i(\mb)}e_i{\ket\mb}$,
\item $K^{-1}f_iK\ket{\mb}=\ib^{-\beta_i(\mb-\boldsymbol{\alpha}_i)}f_i{\ket\mb}$.
\end{enumerate}
\end{lem}

\begin{prop}\label{prop:adjoint of comultiplications}
For $u_i$ $(i\in I$, $u=e,f,k)$, we have
\[
\Delta(\bar{u}_i)\ket{\mb}\ot\ket{\mb'}=\ib^{\Lambda_i(\mb+\mb')}
(K\ot K)^{-1}\Delta^\Gamma(u_i)(K\ot K)\ket{\mb}\ot\ket{\mb'},
\]
on $\W^{\ot 2}$. Here
\[
\Lambda_i(\mb)=\begin{cases}
m_i+m_{i+1}-(\delta_{i0}+\delta_{in})|\mb|-n\delta_{i0}&(u=e)\\
m_i+m_{i+1}+(\delta_{i0}+\delta_{in})(|\mb|+1)-2&(u=f)\\
2m_i-2m_{i+1}&(u=k)
\end{cases},
\]
except when $i=0$ and for $U_q(B^{(1)}(0,n))$, where
\[
\Lambda_0(\mb)=\begin{cases}
2m_1-2|\mb|-2n+1&(u=e)\\
2m_1-2|\mb|-2n+3&(u=f)\\
0&(u=k)
\end{cases}.
\]
Here we should understand $m_0=m_{n+1}=0$.
\end{prop}

\begin{proof}
It follows from Lemmas \ref{app:formula-1} and \ref{app:formula-2}, and the following calculations. For instance, for $i=n$
\begin{align*}
\Delta(\bar{e}_n)|\bf m\rangle\ot|\bf m'\rangle 
&= (1\ot \bar{e}_n + \bar{e}_n\ot \sigma \bar{k}_n)|{\bf m}\rangle\ot|{\bf m}'\rangle \\
&=\kappa(\ib^{1-|\mb'|}[m'_n] |{\bf m}\rangle\ot|{\bf m}'-{\bf e}_n\rangle \\ 
& \quad +(-1)^{|\mb'|}\ib^{2-|\mb|+2m'_n}q^{-2m'_n-1}[m_n]|{\bf m}-{\bf e}_n\rangle\ot|{\bf m}'\rangle),\\
\Delta^\Gamma(e_n)|\bf m\rangle\ot|\bf m'\rangle 
&= (\Gamma^{-1}(1\ot e_n)\Gamma + \Gamma^{-1}(e_n\ot1)\Gamma\cdot(1\ot k_n))|{\bf m}\rangle\ot|{\bf m}'\rangle\\
&=\kappa(\ib^{|\mb|-m_n+1}[m'_n] |{\bf m}\rangle\ot|{\bf m}'-{\bf e}_n\rangle \\
&\quad +\ib^{-|\mb'|+m'_n+2}q^{-2m'_n-1}[m_n]|{\bf m}-{\bf e}_n\rangle\ot|{\bf m}'\rangle),
\end{align*}
and for $i\ne n$
\begin{align*}
\Delta(\bar{e}_i)|\bf m\rangle\ot|\bf m'\rangle 
&= (1\ot \bar{e}_i + \bar{e}_i\ot \bar{k}_i)|{\bf m}\rangle\ot|{\bf m}'\rangle \\
&=\ib^{2m'_{i+1}}[m'_i] |{\bf m}\rangle\ot|{\bf m}'-{\bf e}_i+{\bf e}_{i+1}\rangle \\ 
& \quad+\ib^{2m_{i+1}+2m'_i-2m'_{i+1}}q^{-2m'_i+2m'_{i+1}}[m_i]|{\bf m}-{\bf e}_i+{\bf e}_{i+1}\rangle\ot|{\bf m}'\rangle,\\
\Delta^\Gamma(e_i)|\bf m\rangle\ot|\bf m'\rangle 
&= (\Gamma^{-1}(1\ot e_i)\Gamma + \Gamma^{-1}(e_i\ot1)\Gamma\cdot(1\ot k_i))|{\bf m}\rangle\ot|{\bf m}'\rangle\\
&=\ib^{-m_i-m_{i+1}}[m'_i] |{\bf m}\rangle\ot|{\bf m}'-{\bf e}_i+{\bf e}_{i+1}\rangle \\ 
& \quad+\ib^{m'_i+m'_{i+1}}q^{-2m'_i+2m'_{i+1}}[m_i]|{\bf m}-{\bf e}_i+{\bf e}_{i+1}\rangle\ot|{\bf m}'\rangle.
\end{align*}
\end{proof}

For a quantum affine superalgebras such as $U_q(D^{(2)}_{n+1})$, $U_q(A^{(2)\dagger}_{2n})$, $U_q(C^{(2)}(n+1))$, and $U_q(B^{(1)}(0,n))$, a quantum $R$
matrix $R(z)$ is defined, if it exists, as an intertwiner satisfying 
\[
\check{R}(z)(\pi_x\ot\pi_y)\Delta(u)
=(\pi_y\ot\pi_x)\Delta(u)\check{R}(z),
\]
where $\check{R}(z)=PR(z)$, $P$ is the transposition of the tensor components and
$z=x/y$. We also note that the coproduct we use here is 
\eqref{eq:coproduct for quantum covering}.
For $U_q(D^{(2)}_{n+1})$ or $U_q(A^{(2)\dagger}_{2n})$, the existence 
of quantum $R$ matrices are proved in \cite{KO} or Appendix \ref{app:Adagger}.
We denote them by $\check{R}_{KO}(z)$. Let $\check{R}_{new}(z)$ be the 
quantum $R$ matrices for the quantum groups $U_q(C^{(2)}(n+1))$ or
$U_q(B^{(1)}(0,n))$. From Proposition \ref{prop:adjoint of comultiplications}, we have

\begin{prop}\label{prop:new R matrix}
For generic $x, y\in \Q(q)$, $\check{R}_{new}(z)$ and $\check{R}_{KO}(z)$ have
the following relation:
\[
\check{R}_{new}(z,-q)=(K\ot K)^{-1}\Gamma^{-1}\check{R}_{KO}(z,q)
\Gamma(K\ot K).
\]
\end{prop}

\section{Irreducibility of $\W^{(s)}$}\label{app:irreducibility}

In this appendix, we prove Theorem \ref{thm:irreducibility of higher level osc}.
We adopt the arguments used in the finite-dimensional representations of the quantum affine algebras \cite{KKKO}. We assume that $X=C_n^{(1)}$ and $\W=\W_+$ since the proof for the other two cases are similar.

\subsection{Normalized $R$ matrix}
Let us use the following notations.

\begin{itemize}
\item ${\bf k}$ : the base field, which is the algebraic closure of $\Q(q)$ in $\bigcup_n \C(\!(q^{-1/n})\!)$,

\item $\W[z] = {\bf k}[z^{\pm 1}]\ot_{\bf k} \W(1)$ : the affinization of $\W(1)$, where $z$ is a formal variable (see \cite[Section 4.2]{Kas02}),

\item $\W[z_1]\, \widehat{\ot}\, \W[z_2]$, $\W[z_2]\, \widetilde{\ot}\, \W[z_1]$ : the completions of $\W[z_1]{\ot} \W[z_2]$ and $\W[z_2]{\ot} \W[z_1]$, respectively, where $z_1, z_2$ are formal commuting variables (see \cite[Section 7]{Kas02}),

\item ${\bf k}\llbracket\frac{z_1}{z_2}\rrbracket$ : the ring of formal power series in $\frac{z_1}{z_2}$ over ${\bf k}$, where we regard ${\bf k}(\tfrac{z_1}{z_2}) \subset {\bf k}\llbracket\tfrac{z_1}{z_2}\rrbracket$
under the identification $\frac{1}{1-c(z_1/z_2)}=\sum_{k\ge 0}c^k(\frac{z_1}{z_2})^k\in {\bf k}\llbracket\tfrac{z_1}{z_2}\rrbracket$ for $c\in {\bf k}$.
\end{itemize}

Let $R^{\rm univ}$ be the universal $R$ matrix for $U_q(C_n^{(1)})$. Then we have (cf.\cite[(7.6)]{Kas02})
\begin{equation*}
\xymatrixcolsep{4pc}\xymatrixrowsep{3pc}\xymatrix{
\W[z_1]\, \widehat{\ot}\, \W[z_2] \ \ar@{->}^{R^{\rm univ}}[r] &\ \W[z_2]\, \widetilde{\ot}\, \W[z_1]}.
\end{equation*}

\begin{lem}\label{lem:irreducible}
${\bf k}(\frac{z_1}{z_2})\ot_{{\bf k}[(\frac{z_1}{z_2})^{\pm 1}]}(\W[z_1]\ot \W[z_2])$ is an irreducible module over ${\bf k}(\frac{z_1}{z_2})\ot_{{\bf k}[(\frac{z_1}{z_2})^{\pm 1}]} U_q(C_n^{(1)})[z_1^{\pm 1},z_2^{\pm 1}]$.
\end{lem}
\begin{proof}
See \cite[Proposition 12]{KO}.
\end{proof}

\begin{lem}\label{lem:invariance}
The spaces $\W[z_1]\, \widehat{\ot}\, \W[z_2]$ and $\W[z_2]\,\widetilde{\ot}\, \W[z_1]$ are invariant under the action of ${\bf k}\llbracket\frac{z_1}{z_2}\rrbracket$.
\end{lem}
\begin{proof}
Let us consider $\W[z_2]\,\widetilde{\ot}\,\W[z_1]$. The proof for $\W[z_1]\, \widehat{\ot}\, \W[z_2]$ is the same.
It suffices to check that $\W[z_2]\,\widetilde{\ot}\,\W[z_1]$ is invariant under multiplication by ${\bf k}\llbracket\frac{z_1}{z_2}\rrbracket$.
Let $Q$ be the root lattice for $C_n^{(1)}$ and $Q_+$ the set of non-negative integral linear combinations of simple roots.
By definition of $\W[z_2]\, \widetilde{\ot}\, \W[z_1]$, we have
\[
\W[z_2]\, \widetilde{\ot}\,  \W[z_1]=\sum_{(\lambda,\mu)}F_{(\lambda,\mu)}(\W[z_2]\ot\W[z_1]), 
\]
where 
$F_{(\lambda,\mu)}(\W[z_2]\ot\W[z_1])=\prod_{\beta\in Q_+} \W[z_2]_{\lambda-\beta}\times \W[z_1]_{\mu+\beta}$. Here we understand the weights $\lambda-\beta$ and $\mu+\beta$ of $\W[z_i]$ ($i=1,2$) as elements in the affine weight lattice, say $P$ in \cite[Section 2.1]{Kas02}. We have
\[
(\tfrac{z_1}{z_2})^k F_{(\lambda,\mu)}(\W[z_2]\ot\W[z_1]) \subset F_{(\lambda,\mu)}(\W[z_2]\ot\W[z_1]),
\]
for $k\in\Z_{\ge 0}$, since for a given $\mu+\beta$ ($\beta\in Q_+$), there exist only finitely many $k\in \Z_{\ge 0}$ and $\beta'\in Q_+$ such that $\mu+\beta=\mu + \beta'+k\delta$, where $\delta$ is the null root of $C_n^{(1)}$. This proves the lemma.
\end{proof}

By Lemma \ref{lem:invariance}, we may regard
\begin{equation*}
\begin{split}
&{\bf k}(\tfrac{z_1}{z_2})\ot_{{\bf k}[(\frac{z_1}{z_2})^{\pm 1}]}(\W[z_1]\ot \W[z_2])
\ \subset \
{\bf k}\llbracket\tfrac{z_1}{z_2}\rrbracket \ot_{{\bf k}[(\frac{z_1}{z_2})^{\pm 1}]} \left(\W[z_1]\ot \W[z_2]\right)\
\subset\ \W[z_1]\ \widehat{\ot}\ \W[z_2],\\
&{\bf k}(\tfrac{z_1}{z_2})\ot_{{\bf k}[(\frac{z_1}{z_2})^{\pm 1}]}(\W[z_2]\ot \W[z_1])
\ \subset \
{\bf k}\llbracket\tfrac{z_1}{z_2}\rrbracket \ot_{{\bf k}[(\frac{z_1}{z_2})^{\pm 1}]} \left(\W[z_2]\ot \W[z_1]\right)\
\subset\ \W[z_2]\ \widetilde{\ot}\ \W[z_1].
\end{split}
\end{equation*}
Note that the weight of $z_i^k\ot\ket{{\bf m}}$ is in $-\tfrac{1}{2}\varpi_n +\Z_{\ge 0}\delta$ if and only if $\ket{{\bf m}}=\ket{{\bf 0}}$. Hence one can check without difficulty that
\begin{equation*}\label{eq:coefficint a}
R^{\rm univ} (|{\bf 0}\rangle\ot |{\bf 0}\rangle) = a(\tfrac{z_1}{z_2})\ot( |{\bf 0}\rangle\ot |{\bf 0}\rangle),
\end{equation*}
for some non-zero $a(\frac{z_1}{z_2})\in {\bf k}\llbracket \frac{z_1}{z_2}\rrbracket$, which is  invertible. 
Now, we define the normalized $R$ matrix by
\begin{equation*}
R^{\rm norm}_{z_1,z_2}= a(\tfrac{z_1}{z_2})^{-1}R^{\rm univ}.
\end{equation*}

\begin{lem}\label{lem:nor R = KO R}
The normalized $R$ matrix $R^{\rm norm}_{z_1,z_2}$ gives a ${\bf k}(\frac{z_1}{z_2})\ot_{{\bf k}[(\frac{z_1}{z_2})^{\pm 1}]} U_q(C_n^{(1)})[z_1^{\pm 1},z_2^{\pm 1}]$-linear map
\begin{equation*}\label{eq:normalized R}
\xymatrixcolsep{3pc}\xymatrixrowsep{3.5pc}\xymatrix{
{\bf k}(\frac{z_1}{z_2})\ot_{{\bf k}[(\frac{z_1}{z_2})^{\pm 1}]}(\W[z_1]\ot \W[z_2]) \ \ar@{->}^{R^{\rm norm}_{z_1,z_2}}[r] &\ {\bf k}(\frac{z_1}{z_2})\ot_{{\bf k}[(\frac{z_1}{z_2})^{\pm 1}]}(\W[z_2]\ot \W[z_1]) }, 
\end{equation*}
where $R^{\rm norm}_{z_1,z_2}(|{\bf 0}\rangle\ot |{\bf 0}\rangle)=|{\bf 0}\rangle\ot |{\bf 0}\rangle$.
Moreover, $R^{\rm norm}_{z_1,z_2}$ is a unique such map and hence
$$\check{R}_+(z_1,z_2)=R^{\rm norm}_{z_1,z_2},$$
where $\check{R}_+(z_1,z_2)$ is the quantum $R$ matrix in \eqref{eq:spectral decomp}.
\end{lem}
\begin{proof}
The well-definedness and uniqueness of $R^{\rm norm}_{z_1,z_2}$ follows from Lemma \ref{lem:irreducible} and $R^{\rm norm}_{z_1,z_2}(|{\bf 0}\rangle\ot |{\bf 0}\rangle)=|{\bf 0}\rangle\ot |{\bf 0}\rangle$. In particular, we have $R^{\rm norm}_{z_1,z_2}=\check{R}_+(z_1,z_2)$.
\end{proof} 
 
\subsection{Irreducibility of $\W^{(2)}$}
Let us prove first that $\W^{(2)}$ is irreducible.
Recall that $\W^{(2)}$ is the image of 
\begin{equation*}
\xymatrixcolsep{4pc}\xymatrixrowsep{3pc}\xymatrix{
\W(q^{-1})\ot \W(q) \ar@{->}^{R^{\rm norm}_{q^{-1},q}}[r] & \W(q)\ot \W(q^{-1})
},
\end{equation*}
which is well-defined by \eqref{eq:spectral decomp} and Lemma \ref{lem:nor R = KO R}. 
Let
\begin{equation*}
{\bf K}={\bf k}(\tfrac{z_1}{z_3})\ot{\bf k}(\tfrac{z_2}{z_3}),\quad
{\bf D}={\bf k}[(\tfrac{z_1}{z_3})^{\pm 1},(\tfrac{z_2}{z_3})^{\pm 1}].
\end{equation*}
Consider the following maps
\begin{equation*}
\xymatrixcolsep{-4pc}\xymatrixrowsep{4pc}\xymatrix{
{\bf K} \ot_{{\bf D}}(\W[z_1]\ot \W[z_2]\ot \W[z_3]) 
\ar@{->}_{{\rm id}\ot R^{\rm norm}_{z_2,z_3}}[dr] &  & 
{\bf K} \ot_{{\bf D}}(\W[z_3]\ot \W[z_1]\ot \W[z_2]) 
\\
 & {\bf K} \ot_{{\bf D}}(\W[z_1]\ot \W[z_3]\ot \W[z_2]) 
\ar@{->}_{R^{\rm norm}_{z_1,z_3}\ot {\rm id}}[ur] &
}.
\end{equation*} 
By the hexagon property of $R^{\rm univ}$, we have
\begin{equation}\label{eq:hexagon}
\begin{split}
(R^{\rm norm}_{z_1,z_3}\ot {\rm id})\circ ({\rm id}\ot R^{\rm norm}_{z_2,z_3})
&=(a(\tfrac{z_1}{z_3})^{-1}R^{\rm univ}_{1,3}\ot {\rm id})\circ (a(\tfrac{z_2}{z_3})^{-1} {\rm id}\ot R^{\rm univ}_{2,3})\\
&= a(\tfrac{z_1}{z_3})^{-1}a(\tfrac{z_2}{z_3})^{-1}(R^{\rm univ}_{1,3}\ot {\rm id})\circ ( {\rm id}\ot R^{\rm univ}_{2,3})\\
&= a(\tfrac{z_1}{z_3})^{-1}a(\tfrac{z_2}{z_3})^{-1}R^{\rm univ}_{12,3},
\end{split}
\end{equation}
where the second equality follows from the fact that the action of $z_i$ commutes with those of $e_j$ and $f_j$, and $R^{\rm univ}_{12,3}$ in the last equality is understood as the map
\begin{equation*}
\xymatrixcolsep{4pc}\xymatrixrowsep{3pc}\xymatrix{
(\W[z_1]\ot \W[z_{2}])\ot \W[z_3] \ \ar@{->}^{R^{\rm univ}}[r] &\ \W[z_3]\, \widetilde{\ot}\, (\W[z_1]\ot \W[z_{2}])}.
\end{equation*}
Put
\begin{equation*}\label{eq:R matrix}
R = a(\tfrac{z_1}{z_3})^{-1}a(\tfrac{z_2}{z_3})^{-1}R^{\rm univ}_{12,3}=(R^{\rm norm}_{z_1,z_3}\ot {\rm id})\circ ({\rm id}\ot R^{\rm norm}_{z_2,z_3}).
\end{equation*}
By \eqref{eq:spectral decomp}, we have a well-defined non-zero $U_q(C_n^{(1)})$-linear map ${\bf r}:=R\vert_{z_1=z_2^{-1}=z_3=q}$:
\begin{equation*}
\xymatrixcolsep{7pc}\xymatrixrowsep{3pc}\xymatrix{
\W(q)\ot\W(q^{-1})\ot\W(q)\ \ar@{->}^{{\bf r}}[r] &\ \W(q)\ot\W(q)\ot\W(q^{-1})
}.
\end{equation*} 
By the hexagon property \eqref{eq:hexagon}, we obtain the following. 
\begin{lem}\label{lem:renormal R level 2}
Let $S\subset \W(q)\ot \W(q^{-1})$ be a $U_q(C_n^{(1)})$-submodule. 
Then we have 
$${\bf r}(S \ot \W(q^{-1}))\subset \W(q^{-1})\ot S,$$ 
and the following diagram commutes:
\begin{equation*}
\xymatrixcolsep{7pc}\xymatrixrowsep{3pc}\xymatrix{
S \ot \W(q^{-1}) \ \ar@{->}^{{\bf r}\vert_{S\ot\W(q^{-1})}}[r]\ar@{^{(}->}[d] &  \W(q^{-1})\ot S \ar@{^{(}->}[d]\\
\W(q)\ot\W(q^{-1})\ot\W(q)\ \ar@{->}^{{\bf r}}[r]  &\ \W(q)\ot\W(q)\ot\W(q^{-1})
}
\end{equation*}\qed
\end{lem}

\begin{prop}\label{prop:ireducibility level 2}
$\W^{(2)}$ is irreducible.
\end{prop}
\begin{proof}
Note that $\W(q^{\pm 1})$ is irreducible, and  
${\bf r}= ({\bf r}_{1,3}\ot {\rm id})\ot ({\rm id}\ot {\bf r}_{2,3})$, 
where  
${\bf r}_{2,3}=R^{\rm norm}_{z_2,z_3}\vert_{z_2^{-1}=z_3=q}$ and ${\bf r}_{1,3}=R^{\rm norm}_{z_1,z_3}\vert_{z_1=z_3=q}={\rm id}$. 
Hence, we may apply the same arguments as in \cite[Theorem 3.12]{KKKO} to show that if $S$ is a non-zero submodule of $\W(q)\ot \W(q^{-1})$, then $S$ includes ${\rm Im}(R^{\rm norm}_{q^{-1},q})$. This implies that $\W^{(2)}$ is irreducible.
\end{proof}

\subsection{Proof of Theorem \ref{thm:irreducibility of higher level osc}}
Fix $s\ge 2$.
Let $z_1,\dots,z_s$ be formal commuting variables. Consider
$$
\W[z_1]\ot\dots \ot \W[z_s].
$$
For $a\in \Q(q)^\times$, let $x_i=aq^{2i-1-s}$ and $\W_i = \W(x_i)$ for $1\le i\le s$. 
We have a map ${\bf r}_{i,j}:= R^{\rm norm}_{x_i,x_j}: \W_i\ot \W_j \longrightarrow \W_j\ot \W_i$
for $1\le i<j\le s$ and
\begin{equation*}
\xymatrixcolsep{4pc}\xymatrixrowsep{3pc}\xymatrix{
\W_1\ot \dots\ot  \W_s \ar@{->}^{{\bf r}}[r] & \W_s \ot \dots\ot \W_1},
\end{equation*}
which is the composition of ${\bf r}_{i,j}$ associated to a reduced expression of $w_0\in \mf{S}_s$.  

Let us prove that $\W^{(s)}={\rm Im}({\bf r})$ is irreducible.
Use induction on $s$. It is true for $s=2$ by Proposition \ref{prop:ireducibility level 2}.
Suppose that $s\geq 3$. Let ${\bf r}=\mathbf{r}_s$ be the map in the statement and let $\mathbf{r}_{s-1}$ be the map corresponding to the first $s-1$ factors.
We have the following commutative diagram:
\begin{equation*}
\xymatrixcolsep{5pc}\xymatrixrowsep{2pc}
\xymatrix{
\W_1\ot\cdots\ot \W_s \ar@{->}[r]^{\mathbf{r}_s} \ar@{->}[d]^{\mathbf{r}_{s-1}\ot {\rm id}_{\W_s}} & \ \W_s\ot\cdots\ot \W_1 \\
{\rm Im}(\mathbf{r}_{s-1})\ot \W_s \ar@{^{(}->}[d] & \\
\W_{s-1}\ot\cdots\ot \W_1\ot \W_s  \ar@{->}[ruu]_{\quad \hat{\mathbf{r}}_{s-1}\circ\dots\circ\hat{\mathbf{r}}_1} & \\
}
\end{equation*}
where 
$\hat{\mathbf{r}}_i= {\rm id}^{\otimes s-i-1}\ot \mathbf{r}_{i,s} \ot {\rm id}^{i-1}$ for $1\leq i\leq s-1$. 
Note that $\mathbf{r}_s \neq 0$ since ${\bf r}(\ket{{\bf 0}}^{\ot s})=\ket{{\bf 0}}^{\ot s}$. 
Thus $\mathbf{r}_{s-1}$ has a nonzero image $\W^{(s-1)}$, which is irreducible by the induction hypothesis. 
Applying the hexagon property repeatedly, we have the following commutative diagram:
\begin{equation*}
\xymatrixcolsep{6pc}\xymatrixrowsep{3pc}\xymatrix{
\W^{(s-1)} \ot \W_s \ar@{^{(}->}[d] \ar@{->}^{\mathbf{r}_{\W^{(s-1)},\W_s}}[r] &  \W_s\ot \W^{(s-1)} \ar@{^{(}->}[d] \\
\W_{s-1}\ot\cdots\ot \W_1\ot \W_s \ar@{->}^{\hat{\mathbf{r}}_{s-1}\circ\cdots\circ\hat{\mathbf{r}}_1}[r] & \W_{s}\ot\cdots\ot \W_2\ot \W_1
}.
\end{equation*}
Here the map $\mathbf{r}_{\W^{(s-1)},\W_s}$ is given by
\begin{equation*}
\mathbf{r}_{\W^{(s-1)},\W_s}= c R^{\rm univ}_{1\dots s-1,s}\vert_{z_i=aq^{2i-1-s}},
\end{equation*}
where $c$ is an element in $\bigotimes_{i<j}{\bf k}(\frac{z_i}{z_j})$, and $R^{\rm univ}_{1\dots s-1,s}$ is the universal $R$ matrix 
\begin{equation*}
\xymatrixcolsep{4pc}\xymatrixrowsep{3pc}\xymatrix{
(\W[z_1]\ot\dots\ot \W[z_{s-1}])\ot \W[z_s] \ \ar@{->}^{R^{\rm univ}}[r] &\ \W[z_s]\, \widetilde{\ot}\, (\W[z_1]\ot\dots\ot \W[z_{s-1}])}.
\end{equation*}
Thus the image of $\mathbf{r}_s$ is equal to that of $\mathbf{r}_{\W^{(s-1)},\W_s}$.

Now, let $S$ be a non-zero submodule of $\W_s\ot\W^{(s-1)}$. 
Put ${\bf r}'={\bf r}_{s,s}\circ\mathbf{r}_{\W^{(s-1)},\W_s}$.
Then as in Lemma \ref{lem:renormal R level 2}, we can check the following commutative diagram:
\begin{equation*}
\xymatrixcolsep{7pc}\xymatrixrowsep{3pc}\xymatrix{
S \ot \W_s \ \ar@{->}^{{\bf r}'\vert_{S\ot\W_s}}[r]\ar@{^{(}->}[d] &  \W_s\ot S \ar@{^{(}->}[d]\\
\W_s\ot\W^{(s-1)}\ot\W_s\ \ar@{->}^{{\bf r}'}[r]  &\ \W_s\ot\W_s\ot\W^{(s-1)}
}
\end{equation*}
Again by the same arguments as in \cite[Theorem 3.12]{KKKO}, we conclude that $\W^{(s)}\subset S$, which implies that $\W^{(s)}$ is irreducible. This completes the proof.

\begin{rem}{\rm
The universal $R$ matrix for $C^{(2)}(n+1)$ and $B^{(1)}(0,n)$ can be found in \cite{H99}.
}
\end{rem}

\end{document}